\documentclass[12pt,a4paper]{article}
\usepackage{subfiles}
\usepackage{paper}

\setbool{noProblems}{false}
\setbool{noChanges}{true}
\newcommand{\dobib}{\bibliographystyle{plain}\bibliography{Diss.bib}}

\title{A General Theorem of Gauß Using Pure Measures}
\author{Moritz Schönherr, Friedemann Schuricht\\ \small TU Dresden - Fachrichtung
Mathematik \\ \small 01062 Dresden, Germany}
\begin{document}
\renewcommand{\dobib}{}

\maketitle

\section*{Abstract}
This paper shows that finitely additive \tme{}s occur naturally in very
general Divergence Theorems. The main results are two such theorems. The first
proves the existence of \tpfa{} \tggtme{}s for sets of finite perimeter, which
yield a Gauß formula for essentially bounded vector fields having divergence
measure. The second extends a result of Silhavy \cite{silhavy_divergence_2009}
on normal traces. In particular, it is shown that a Gauß Theorem for unbounded
vector fields having divergence measure necessitates the use of \tpfa{} \tme{}s
acting on the gradient of the scalar field. All of these \tme{}s are shown to
have their \tcor{} on the boundary of the domain of integration.

\section{Introduction}
The Divergence Theorem, or Theorem of Gauß, is a very important theorem in real
analysis.  It connects the integral over a volume with the integral over the
bounding surface of said volume. Its classic form for smooth $\dom \subset \Rn$
and smooth vector fields
$\funv : \dom \to \Rn$ is
\begin{equation*}
	\I{\dom}{\divv{\funv}}{\lem} = \I{\bd{\dom}}{\funv \cdot
	\normal{}}{\ham^{n-1}} \,,
\end{equation*}
where $\lem$ is the Lebesgue measure, $\ham^{n-1}$ the $n-1$-dimensional
Hausdorff measure and $\normal{}$ the unit outer normal to $\dom$. 
In the theory of Continuum Mechanics, but
also in general analytical problems it is desirable to have the theorem at hand
for very general vector fields and very general domains of integration. The
main problem is to find a useful substitute for the area integral. Up to now,
the area integral was substituted for a continuous linear functional on a
function space on the boundary (cf. Silhavy \cite{silhavy_divergence_2009}) or
for a normal trace given by a measure on the boundary, in the case of the
vector fields considered being essentially bounded (cf. Chen
\cite{chen_gauss-green_2009},
\cite{chen_divergencemeasure_1999},\cite{chen_theory_2001},
\cite{chen_extended_2003})).
In Schuricht \cite{schuricht_new_2007}, a general limit formula is proved.

In this paper, it will be shown that \tpfa{} \tme{}s are a natural substitute
for area measure, when vector fields having divergence measure are considered.
There are two main results. The first one is a theorem for essentially bounded
vector fields having divergence measure. It is proved that for sets of finite
perimeter there is a normal measure which gives rise to a very general Gauß
formula. The second theorem treats unbounded vector fields and open sets. A
result due to Silhavy \cite{silhavy_divergence_2009} is extended and it is
proved that a Gauß-Green formula for these vector fields necessarily contains a
\tpfa{} \tme{} acting on the values of the gradient of the scalar field near
the boundary.

The first section is a primer on \tpfa{} \tme{}s. Many properties and
integration with respect to \tpfa{} \tme{}s are laid out.

The second section contains the results for essentially bounded vector fields.

The third section contains the theorem for unbounded vector fields and open
sets.

Concerning notation, in the following $n\in \N$ denotes a positive natural
number and $\Rn$ the vector space of real $n$-tuples. For a set $\dom \subset
\Rn$ the set $\dnhd{\dom}{\delta}$ denotes the open $\delta$-neighbourhood of
$\dom$. Open balls with radius $\delta >0$ and centre $\eR \in \Rn$ are written
$\ball{\eR}{\delta} = \dnhd{\{\eR\}}{\delta}$. The Borel subsets of $\dom$, i.e. the \tsme{} generated by all
relatively open sets in $\dom$, is denoted by $\bor{\dom}$. $\lem$ is the Lebesgue measure and $\ham^d$ the $d$-dimensional
Hausdorff measure. For set function $\me$ on $\dom$, $\reme{\me}{\als}$ denotes
the restriction of $\me$ to $\als$. The Banach space of equivalence classes of
$p$-integrable functions is denoted by $\lpbld{p}{\dom}$ and $\hocon{p}$
denotes the Hölder-conjugate of $p$. (Weak) Derivates of functions $\fun$ are written
$\Deriv{\fun}$. The divergence of a vector field $\funv$, be it
classical or distributional, is denoted by $\divv{\funv}$.

\section{A Short Primer On Pure Measures}
In this article, set functions $\me : \al \subset \pos{\dom} \to \R$ will be called
\textit{\tme{}}, if for all $m \in \N$ and every pairwise disjoint $\{\als_k\}_{k=1}^m\subset \al$ with
$\bigcup \limits_{k=1}^m \als_k \in \al$
\begin{equation*}
	\me\left(\bigcup \als_k\right) = \sum \limits_{k=1}^m \me(\als_k) \,.
\end{equation*}
If this holds with $m= \infty$, the \tme{} is called \tsme{}. A \tme{} is
called \textit{bounded} if 
\begin{equation*}
	\sup \limits_{\als \in \al} |\me(\als)| < \infty \,.
\end{equation*}
An \textit{\tal{}} is a class of sets which is stable under union, intersection and
differences and contains at least $\emptyset$.
The spaces of \tme{}s considered in this thesis are defined in accordance with \cite{rao_theory_1983}.

\begin{definition}
	Let $\dom \subset \Rn$ and $\al \subset \pos{\dom}$ be an algebra. The set of all bounded \tme s $\me : \al \to \R$ is denoted by
	\begin{equation*}
		\baa  \,.
		\nomenclature[m]{$\baa$}{space of bounded \tme{}s on $\dom$}
	\end{equation*}
	The set of all bounded \tsme{}s $\sme : \al \to \R$ is denoted by 
	\begin{equation*}
		\caa \,.		
		\nomenclature[m]{$\caa$}{space of bounded \tsme{}s on $\dom$}
	\end{equation*}
\end{definition}
The following proposition is an application of Riesz's Decomposition Theorem
 (cf. \cite[p. 241]{rao_theory_1983}). 
\begin{proposition}\label{prop:yosh_hew_dec}
	Let $\dom \subset \Rn$ and $\al \subset \pos{\dom}$ be an algebra. Then
	every $\me \in \baa$ can uniquely be decomposed into $\me_c \in \caa$
	and $\me_p \in \baa$ such that
	\begin{equation*}
		\me = \me_c + \me_p
	\end{equation*}
	\nomenclature[m]{$\me_c$}{$\sigma$-additive part of
	$\me$}\nomenclature[m]{$\me_p$}{\tpfa{} part of $\me$}
	and for every $\sme \in \caa$
	\begin{equation*}
		0 \leq \sme \leq |\me_p| \implies \sme = 0 \,.
	\end{equation*}
\end{proposition}
\begin{definition}
	Let $\dom \subset \Rn$ and $\al \subset \pos{\dom}$ be an \tal{}. Then
	every such \tme{} $\me_p$ is called
	\textbf{\tpfa}\index{pure@\tpfa{} \tme{}}.
	Notice that $\me_p$ is not $\sigma$-additive, by definition.
\end{definition}
One important example of \tme{}s that are \tpfa{} are density
\tme{}s. The following new example presents a particular \tdme{}, namely a
density at zero. In the literature, examples of \tpfa{} \tme{} are only known
for $\dom = \N$ (cf. \cite[p. 247]{rao_theory_1983}), they are defined on very
small \tal{}s (cf. \cite[p. 246]{rao_theory_1983}) or they are constructed in such a way that the \tme{} cannot be
computed explicitly, even on simple sets (cf. \cite[p.
57f]{yosida_finitely_1951}). The example given here is constructed on $\dom =
\Rn$ and lives on the Borel subsets of $\dom$.

\begin{example}\label{ex:dzero}
	Let $\dom := \ball{0}{1}\subset \Rn$ be open. Then there exists $\me
	\in \baA{\bor{\dom}}$, $\me \geq 0$ such that
	for every $\bals \in \bor{\dom}$
	\begin{equation*}
		\me(\bals) = \lim \limits_{\delta \downarrow 0} \frac{\lem(
		\bals \cap \ball{0}{\delta})}{\lem(\ball{0}{\delta})}
	\end{equation*}
	if this limit exists. This \tme{} is non-unique. Its existence is 
	shown Proposition 5.7 in \cite[p. 25]{schonherr_pure_2017} (take $\lambda:= \lem$ and
	$\ferm = \{0\}$). 
		
	It is shown in Example \ref{ex:dzero_pfa}
	that $\me$ is indeed \tpfa{}. Figure \ref{fig:dzero_pfa} shows the
	family $\{\als_k\}_{k\in \N} \subset \bor{\dom}$
	\begin{equation*}
		\als_k := \left[\frac{1}{k+2},\frac{1}{k+1}\right) \times
		[-1,1]^{n-1} \,.
	\end{equation*}
	For this family
	\begin{equation*}
		\sum \limits_{k\in \N} \me(\als_k \cap \dom) = 0 \neq
		\me\left(\left(0,\frac{1}{2}\right)\times [-1,1]^{n-1}\cap \dom \right) = \me
		\left(\bigcup\limits_{k =1}^\infty \als_k \cap \dom\right)\,.
	\end{equation*}
	Hence, $\me$ is not a \tsme{}.
\end{example}
\begin{figure}[H]
	\centering
	\begin{tikzpicture}[scale=1.2]
		\draw (-0.2,-0.1) node {x};
		\draw[gray,thick,->] (0,0) -- node[above] {$\delta$} (150:2.2);
		\draw[fill] (0,0) circle [radius = 0.01];
		\draw[gray,dashed] (0,0) circle [radius=2.2];
		\draw (0.1,-2) rectangle (0.2,2);
		\draw (0.2,-2) rectangle (0.4,2);
		\draw (0.4,-2) rectangle (0.7,2);
		\draw (0.7,-2) rectangle (1.0,2);
		\draw (1.0,-2) rectangle node {$A_k$}(1.5,2);
		\draw (1.5,-2) rectangle node {...} (2.1,2);
		\draw (2.1,-2) rectangle node {$A_2$} (2.8,2);
		\draw (2.8,-2) rectangle node {$A_1$} (3.8,2);
	\end{tikzpicture}
	\caption{A family of sets on which $\me$ is not
	$\sigma$-additive}\label{fig:dzero_pfa}
\end{figure}
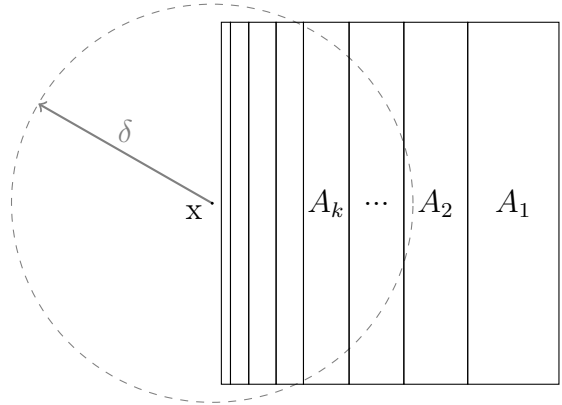
Measures that do not charge sets of Lebesgue measure
zero are of special interest, because these \tme{}s lend themselves naturally to the integration of
functions that are only defined outside of a set of measure zero. When treating non $\sigma$-additive measures, one carefully has to
distinguish the following two notions (cf. \cite[p. 159]{rao_theory_1983}).
\begin{definition}
	Let $\dom \subset \Rn, \al \subset \pos{\dom}$ be an algebra and
	$\metoo\in \baa$. Then $\me \in \baa$ is called 
	\begin{enumerate}
		\item \textbf{\tac}\index{absolutely@\tac{}} with respect to
			$\metoo$, if for every $\eps > 0$ there exists $\delta > 0$ such that for all $\als \in \al$
			\begin{equation*}
				|\metoo|(\als) < \delta \implies |\me(\als)| < \eps \,.
				\nomenclature[m]{$\me \ac \metoo$}{$\me$ is \tac{} w.r.t.
			$\metoo$}
			\end{equation*}
			In this case, write $\me \ac \metoo$.
		\item \textbf{\twac}\index{weakly absolutely@\twac{}} with
			respect to $\metoo$, if for every $\als \in \al$
			\begin{equation*}
				|\metoo|(\als) = 0 \implies \me(\als) = 0 \,.
				\nomenclature[m]{$\me \wac \metoo$}{$\me$ is \twac{}
			w.r.t. $\metoo$}
			\end{equation*}
			In this case, write $\me \wac \metoo$.
	\end{enumerate}
	The set of all \twac{} \tme{}s in $\baa$ is denoted by 
	\begin{equation*}
		\baaw{\metoo} \,.
		\nomenclature[m]{$\baaw{\metoo}$}{bounded and \twac{} measures
		w.r.t. $\metoo$}
	\end{equation*}
\end{definition}
As in Proposition \ref{prop:yosh_hew_dec}, \twac{} \tme{}s can be decomposed
into \tpfa{} and $\sigma$-additive parts.
\begin{proposition}\label{prop:dec_pfa_wac}
	Let $\dom \subset \Rn$, $\al \subset \pos{\dom}$ be an algebra and
	$\metoo \in \baA{\al}$.

	Then for every $\me \in \baaw{\metoo}$ there exist unique $\me_c \in \caa
	\cap \baaw{\metoo}$, $\me_p \in \baaw{\metoo}$ such that
	\begin{equation*}
		\me = \me_c + \me_p 
	\end{equation*}
	such that for all $\sme \in \caa$
	\begin{equation*}
		0 \leq \sme \leq \tova{\me_p} \implies \sme = 0 \,.
	\end{equation*}
\end{proposition}
\begin{proof}
	See Proposition 3.7 in \cite[p. 8]{schonherr_pure_2017}.	
\end{proof}
\begin{example}\label{ex:dzero_pfa}
	Since the \tme{} $\me$ from Example \ref{ex:dzero} is positive and
	$\me_c \perp \me_p$, using the additivity of the total variation on
	orthogonal element (cf. \cite[p. 25]{rao_theory_1983},
	\cite[p. 3]{schonherr_pure_2017})  yields
	\begin{equation*}
		0 \leq \tova{\me_c} \leq \tova{\me_c} + \tova{\me_p} = \tova{\me} = \me \,.
	\end{equation*}
	Hence, for every $\delta > 0$
	\begin{equation*}
		\tova{\me_c}(\ball{0}{\delta}^c)  = 0 \,.
	\end{equation*}
	Thus
	\begin{equation*}
		\tova{\me_c}(\dom \setminus \{0\}) = \lim \limits_{\delta \downarrow 0}
		\tova{\me_c}(\ball{0}{\delta}^c) = 0  \,.
	\end{equation*}
	But $\tova{\me_c}(\{0\}) \leq \me(\{0\})=0$. Hence
	\begin{equation*}
		\tova{\me_c}(\dom) = 0 
	\end{equation*}
	and $\me = \me_p$ is \tpfa{}.
\end{example}
The structure of $\me_p$ is described by the following proposition taken from
\cite[p. 244]{rao_theory_1983} (cf. \cite[p. 56]{yosida_finitely_1951}).

\begin{proposition}\label{prop:pfa_iff_supseq}
	Let $\dom\subset \Rn$, $\sal \subset \pos{\dom}$ be a \tsal{} and $\sme
	\in \cas$, $\sme \geq 0$.  Then $\me \in \basws$ is \tpfa{} if and only if there exists a decreasing sequence $\{A_k\}_{k\in \N}\subset \sal$ such that
	\begin{equation*}
		\sme(\als_k) \xrightarrow{k \to \infty} 0
	\end{equation*}
	and for all $k \in \N$
	\begin{equation*}
		|\me_p|(\als_k^c) = 0 \,.
	\end{equation*}
\end{proposition}
\begin{definition}
	Let $\dom \subset \Rn$, $\sal \subset \pos{\dom}$ be a \tsal{}, $\sme
	\in \cas$, $\sme \geq 0$ and $\me_p \in \basws$ be \tpfa{}. Then every $\als\in \sal$ such that
	\begin{equation*}
		|\me_p|(\als^c) = 0
	\end{equation*}
	is called \textbf{\taur}\index{aura@\taur} of $\me_p$. 

	Any decreasing sequence $\{\als_k\}_{k\in \N} \subset \sal$ of auras
	for $\me_p$ such that 
	\begin{equation*}
		\sme(\als_k) \xrightarrow{k\to \infty} 0
	\end{equation*}
	is called \textbf{\tsupseq}\index{aura seq@\tsupseq{}}.
\end{definition}

Intuitively speaking, \twac{} \tme{}s are \tpfa{} if and only if they concentrate in the vicinity of a set of \tme{} zero.
Reviewing Example \ref{ex:dzero}, the support (cf. \cite[p.30]{Ambrosio2000}) of the \tme{} can be seen to lie
outside of $\dom\setminus \{0\}$. Yet the construction of the \tme{} would
still work on this set. Hence, it is possible for a \tpfa{} \tme{} to have support outside
of its domain of definition. This necessitates the following definition of
\tcor{}.

\begin{definition}
	Let $\dom \subset \Rn$, $\al \subset \pos{\dom}$ be an \tal{} containing every relatively open set in $\dom$. Furthermore let $\me \in \baa$. Then the set
\begin{equation*}
	\cor{\me} := \{ \eR \in \Rn\setpipe |\me|(\nhd \cap
	\dom) > 0, \forall \nhd \subset \Rn, \nhd \text{ open}, \eR\in \nhd\} 
	\end{equation*}
	\nomenclature[m]{$\cor{\me}$}{\tcor{} of $\me$}
	is called \textbf{\tcor}\index{core@\tcor{}} of $\me$.

	 Let $d\in [0,n]$ be the Hausdorff dimension of $\cor{\me}$. Then $d$
	 is called \textbf{\tcordimo{\me}}\index{core dimension@\tcordimo{\me}} and $\me$ is called \textbf{\tdd{d}}.
\end{definition}
\begin{proposition}\label{prop:corenul_pfa}
	Let $\dom \in \bor{\Rn}$ and $\me \in \bawl{\dom}$.

	If $\cor{\me} \cap \dom$ is a $\lem$-\tnulset{} then $\me$ is \tpfa{}.
\end{proposition}

Now, integration with respect to \tme{} which are not necessarily
$\sigma$-additive is outlined.
Measurability of functions is not defined through the regularity of preimages but by
approximability by \tsimf{} functions in measure. In this definition, the
\tme{} is needed on possibly non-measurable sets. Hence, an \toume{}
has to be used. This \toume{} is defined as in the case of \tsme{}s (cf.
\cite[p. 86]{rao_theory_1983}, \cite[p. 42]{halmos_measure_1974}).
\begin{definition}
	Let $\dom \subset \Rn$ and $\al \subset \pos{\dom}$ be an \tal{}. For
	$\me \in \baa$, $\me \geq 0$ the \textbf{\toume} of $\me$ is defined
	for $\bals \in \pos{\dom}$ by
	\begin{equation*}
		\oume{\me}(\bals) := \inf\limits_{\substack{\als\in \al,\\\bals
		\subset \als}}\me(\als) \,.
		\nomenclature[m]{$\oume{\me}$}{\toume{} for $\me$}
	\end{equation*}
\end{definition}

Now, convergence in measure can be defined. 
The definition is taken from \cite[p. 92]{rao_theory_1983} (cf.
\cite[p. 91]{halmos_measure_1974}).
\begin{definition}
	Let $\dom \subset \Rn$ and $\al \subset \pos{\dom}$ be an \tal{} and
	$\me: \al \to \R$ be a \tme{}. A sequence $\{\fun_k\}_{k \in \N}$ of
	functions $\fun_k : \dom \to \R$ is said to
	converge \textbf{\tconvim{}}\index{convergence \tconvim{}} to a function $\fun : \dom \to \R$ if for every $\eps > 0$
	\begin{equation*}
		\lim \limits_{k \to \infty} \oume{\tova{\me}} \{\eR \in \dom \setpipe |\fun_k(\eR) - \fun(\eR) | > \eps \} = 0 \,.
	\end{equation*}
	In this case, write
	\begin{equation*}
		\fun_k \convim{\me} \fun \,.
		\nomenclature[i]{$\fun_k \convim{\me} \fun$}{$\fun_k$ converge
		\tconvim{} to $\fun$}
	\end{equation*}
	\end{definition}
Note that the limit in measure is not unique, yet. Therefore, the following
notion of equality almost everywhere is needed.
The definition is taken from \cite[p. 88]{rao_theory_1983}.
\begin{definition}
	Let $\dom \subset \Rn$, $\al \subset \pos{\dom}$ and $\me: \al \to \R$ be a measure.

	Then $\fun: \dom \to \R$ is called \textbf{\tnulfun}\index{null
	function@\tnulfun{}}, if for every $\eps >0$
	\begin{equation*}
		\oume{\tova{\me}} \left(\left\{\eR \in \dom \setpipe |\fun(\eR)|>\eps\right\}\right) = 0 \,.
	\end{equation*}

	Two functions $\fun_1: \dom \to \R$, $\fun_2: \dom \to \R$ are called
	\textbf{equal almost everywhere (\tmae{})}\index{equal \tmae{}} with respect to $\me$, if $\fun_1 - \fun_2$ is a \tnulfun{}.

	In this case, write
	\begin{equation*}
		\fun_1 = \fun_2 \mae{\me}
		\nomenclature[i]{$\fun_1 = \fun_2 \mae{\me}$}{$\fun_1 = \fun_2$ \tmae{}}
	\end{equation*}
\end{definition}
\begin{remark}
	If $\fun: \dom \to \R$ is a \tnulfun{}, then it need not be true that
	\begin{equation}\label{eq:nullfunc}
		\oume{\tova{\me}} \left (\left\{\eR\in \dom \setpipe \fun (\eR) \neq 0 \right\}\right ) = 0 \,.
	\end{equation}
	Take e.g. the \tdme{} $\me$ introduced in Example \ref{ex:dzero} and
	$\fun(\eR) := |\eR|$. Then $\fun$ is a \tnulfun{} but
	\begin{equation*}
		\oume{\tova{\me}}(\{\eR\in \Rn | \fun(\eR)\neq 0\} =
		\me(\ball{0}{1}\setminus\{0\}) = 1 > 0 \,.
	\end{equation*}

	This entails that the notion of equality almost everywhere that was
	defined above does not imply the existence of a \tnulset{} such that
	$\fun_1 = \fun_2$ outside of that set. Take e.g. the \tdme{} introduced
	in Example \ref{ex:dzero}, $\fun_1(\eR) := |\eR|$ and
	$\fun_2(\eR):=2\fun_1(\eR)$.

	On the other hand, if $\me$ is a \tsme{} and $\al$ a \tsal{}, then Equation \eqref{eq:nullfunc} is equivalent to $\fun$ being a \tnulfun{} (cf. \cite[p. 89]{rao_theory_1983}).

\end{remark}
The limit in measure turns out to be unique in the sense of almost equality.
This is stated in the following proposition taken from \cite[p. 92]{rao_theory_1983}.
\begin{proposition}
	Let $\dom \subset \Rn$, $\al \subset \pos{\dom}$ be an \tal{} and $\me : \al \to \R$ be a measure. 
	Furthermore let $\{\fun_k\}_{k \in \N}$ be a sequence of functions
	$\fun_k : \dom \to \R$ and $\fun, \tilde{\fun} : \dom \to \R$ be functions such that
	\begin{equation*}
		\fun_k \convim{\me} \fun\,.
	\end{equation*}
	Then
	\begin{equation*}
		\fun_k \convim{\me} \tilde{\fun} \iff	\fun = \tilde{\fun} \mae{\me}
	\end{equation*}
\end{proposition}
Now, the notion of measurability is introduced.
The definition is similar to the definition of $T_1$-measurability in \cite[p. 101]{rao_theory_1983}.
\begin{definition}
	Let $\dom \subset \Rn$ and $\al \subset \pos{\dom}$ be an \tal{} and
	$\me : \al \to \R$ be a \tme{}. A function $\fun : \dom \to \R$ is
	called \textbf{\tmefun{}}\index{measurable function@\tmefun{} function}
	if there exists a sequence $\{\simf_k\}_{k \in \N}$ of \tsimf{}
	functions $\simf_k : \dom \to \R$ such that
	\begin{equation*}
		\simf_k \convim{\me} \fun \,.
	\end{equation*}
\end{definition}
The integral for \tmefun{} functions can now be defined via $\lp^1$-Chauchy
sequences. This is of course well-defined (cf. \cite[p. 102]{rao_theory_1983}).
\begin{definition}
	Let $\dom \subset \Rn$, $\al \subset \pos{\dom}$ be an \tal{} and $\me:
	\al \to \R$ be a \tme{}. A function $\fun : \dom \to \R$ is said to be
	\textbf{\tIfun{}}\index{integrable function@\tIfun{} function} if there
	exists a sequence $\{\simf_k\}_{k\in \N}$ of \tIfun{} \tsimf{}
	functions $\simf_k : \dom \to \R$ such that
	\begin{enumerate}
		\item $\simf_k \convim{\me} \fun$.
		\item $\lim \limits_{k,l \to \infty} \I{\dom}{|\simf_k - \simf_l|}{|\me|} = 0$.
	\end{enumerate}
	In this case, denote
	\begin{equation*}
		\I{\dom}{\fun}{\me} := \lim\limits_{k \to \infty} \I{\dom}{\simf_k}{\me} \,.
		\nomenclature[i]{$\I{\dom}{\fun}{\me}$}{\tI{} of $\fun$ w.r.t. \tme{} $\me$}
	\end{equation*}
		The sequence $\{\simf_k\}_{k\in \N}$ is called
	\textbf{\tdetseq{}}\index{determining seq@\tdetseq{} of an \tIfun{}
	function} for the \tI{} of $\fun$. 
\end{definition}
\begin{remark}
	In particular, \tIfun{} functions are \tmefun{}.
	This notion of \tI{} is also called Daniell-Integral in the literature
	(cf. \cite{rao_theory_1983}).
\end{remark}
The $\lp^p$-spaces are defined in the usual way (cf. \cite[p. 121]{rao_theory_1983}).
\begin{definition}
	Let $\dom \subset \Rn$, $\al \subset \pos{\dom}$ be an \tal{}, $\me: \al \to \R$ be
	a measure and $p\in [1,\infty)$. Then the set of all \tmefun{} functions $\fun : \dom \to \R$ such that $|f|^p$ is $|\me|$-\tIfun{} is denoted by
	\begin{equation*}
		\Lpam{p}\,.
		\nomenclature[f]{$\Lpam{p}$}{$p$-\tIfun{} functions w.r.t to \tal{} $\al$ and \tme{} $\me$}
	\end{equation*}
	If $\al = \bor{\dom}$, write
	\begin{equation*}
		\Lpbm{p}\,.
		\nomenclature[f]{$\Lpbm{p}$}{$\Lpam{p}$ with $\al = \bor{\dom}$}
	\end{equation*}
	For $\fun_1, \fun_2 \in \Lpam{p}$
	\begin{equation*}
		\fun_1 = \fun_2 \mae{\me} 
	\end{equation*}
	defines an equivalence relation.
	The set of all equivalence classes of this relation is denoted by
	\begin{equation*}
		\lpam{p}\,.
		\nomenclature[f]{$\lpam{p}$}{Equivalence classes in $\Lpam{p}$}
	\end{equation*}
	If $\al = \bor{\dom}$, write
	\begin{equation*}
		\lpbm{p}\,.
		\nomenclature[f]{$\lpbm{p}$}{$\lpam{p}$ with $\al = \bor{\dom}$}
	\end{equation*}
\end{definition}
\begin{definition}
	Let $\dom \subset \Rn$, $\al \subset \pos{\dom}$ be an \tal{} and $\me: \al \to \R$ a \tme{}. Then for every $p \in [1,\infty)$ and $\fun \in \Lpam{p}$ write
	\begin{equation*}
		\norm{p}{\fun} := \left(\I{\dom}{|\fun|^p}{|\mu|}\right)^{\frac{1}{p}} \,.
		\nomenclature[n]{$\norm{p}{\fun}$}{$L^p$-norm of $\fun$}
	\end{equation*}
	Furthermore, for \tmefun{} $\fun : \dom \to \R$ define
	\begin{equation*}
		\essup{}{\fun} := \inf \left\{ \K \in \R \setpipe |\me|^*\left(\{\eR\in \dom | \fun(\eR) > \K\}\right) = 0\right\}
	\end{equation*}
	and
	\begin{equation*}
		\normi{\fun} := \essup{}{|\fun|} \,.
	\end{equation*}
	The set of all \tmefun{} functions $\fun : \dom \to \R$ such that
	\begin{equation*}
		\normi{\fun} < \infty
	\end{equation*}
	is denoted by
	\begin{equation*}
		\Liam \,.
	\end{equation*}
	As in the case $p\in [1,\infty)$, 
	\begin{equation*}
		\liam
	\end{equation*}
	denotes the set of all equivalence classes in $\Liam$ with respect to
	equality almost everywhere.

	In the case $\al = \bor{\dom}$, only write
	\begin{equation*}
		\Libm \text{ and } \libm \text{ respectively.}
	\end{equation*}
\end{definition}
The integral defined in this way shares many properties of the
Lebesgue-integral. The Hölder and Minkwoski inequality hold true. Furthermore,
dominated convergence is available when using convergence in \tme{} instead of
pointwise convergence (cf. \cite[p. 105ff]{rao_theory_1983}). 

Before proceeding to the characterisation of the dual of $\lp^\infty$, a new
integral symbol is introduced, which gives formulas traces
and integrals over \tpfa{} \tme{}s a more pleasing
shape.
\begin{definition}
	Let $\dom \subset \Rn$ be bounded and $\ferm \subset \cl{\dom}$ be
	closed. Then for every $\me \in \bawl{\dom}$ such that
	\begin{equation*}
		\cor{\me} \subset \ferm,
	\end{equation*}
	every $\fun \in \lpbm{1}$ and $\delta >0$ write
	\begin{equation*}
		\sI{\ferm}{\fun}{\me} :=
		\I{\dnhd{\ferm}{\delta}\cap\dom}{\fun}{\me} \,.
		\nomenclature[i]{$\sI{\ferm}{\fun}{\me}$}{$\I{\dnhd{\ferm}{\delta}}{\fun}{\me}$,
		where $\ferm = \cor{\me}$}
	\end{equation*}
\end{definition}
\begin{remark}
	This notion of integral is well-defined since the definition of
	$\cor{\me}$ yields
	\begin{equation*}
		\tova{\me}((\dnhd{\ferm}{\delta})^c) = 0 
	\end{equation*}
	for any $\delta >0$.
\end{remark}
The following proposition is a specialised version of the proposition from
\cite[p. 139]{rao_theory_1983} (cf. \cite[p. 53]{yosida_finitely_1951}).
\begin{proposition}\label{prop:dual_liss_pre}
	Let $\dom \subset \Rn$, $\sal \subset \pos{\dom}$ be a \tsal{} and $\sme : \sal \to \R$ be a \tsme{}.

	Then for every $\de \in \dual{\left( \liss \right)}$ there exists a unique $\me \in \basws$ such that
	\begin{equation*}
		\df{\de}{\fun} = \I{\dom}{\fun}{\me}
	\end{equation*}
	for every $\fun \in \liss$ and
	\begin{equation*}
		\norm{}{\de} = \norm{}{\me} = \tova{\me}(\dom) \,.
	\end{equation*}
	On the other hand, every $\me \in \basws$ defines $\de\in
	\dual{\liss}$.

	Hence, $\dual{\liss}$ and $\basws$ can be identified.
\end{proposition}
Using the decomposition Proposition \ref{prop:dec_pfa_wac} from above, one obtains a more refined characterisation of the dual of
$\liss$. In particular, every element of the dual space is the sum of a \tsme{}
with $\lem$-density and a \tpfa{} \tme{}. In contrast to the literature, this
makes the intuitive idea of the dual of $\lp^\infty$ being $\lp^1$ plus
something which is not \twac{} with respect to Lebesgue measure precise.
\begin{proposition}\label{thm:dual_liss}
	Let $\dom \subset \Rn$ and $\sal \subset \pos{\dom}$ be a \tsal{} and $\sme : \sal \to \R$ be a \tsme{}. Then for every $\dual{u} \in \dual{\liss}$ there exists a unique \tpfa{} $\me_p \in \basws$ and a unique $\dfun \in \lpss{1}$ such that
	\begin{equation*}
		\df{\dual{u}}{\fun} = \I{\dom}{\fun\dfun}{\lem} + \I{\dom}{\fun}{\me_p}
	\end{equation*}
	for every $\fun \in \liss$.
\end{proposition}
\begin{proof}
	See Theorem 4.14 in \cite[p. 21]{schonherr_pure_2017}.
\end{proof}
\begin{remark}
	Note that the $\lp$-space over a \tme{} $\me\geq 0$ is in general not
	complete. Nevertheless, the completion is known to be the set of all
	\tac{} \tme{}s whose $p$-norm is finite, i.e. all bounded \tme{}s
	$\metoo$ with $\metoo \ac \me$ and
	\begin{equation*}
		\lim\limits_{\partition \in \partitions} \sum
		\limits_{\substack{\parte\in \partition\\\me(\parte) \neq 0}} \left |
		\frac{\metoo(\parte)}{\me(\parte)}\right|^p \me(\parte) <
		\infty \,.
	\end{equation*}
	Here, the limit is taken over the directed set $\partitions$ of all partitions
	$\partition$ of $\dom$. See \cite[p. 185ff]{rao_theory_1983} for
	reference. Using the convention $\frac{0}{0} = 0$, this limit is the
	same as the \textit{refinement integral}
	\begin{equation*}
		\RI{\dom}{\left | \frac{\metoo}{\me}\right|^p \me}
	\end{equation*}
	as defined by Kolmogoroff in \cite{Kolmogoroff1930}.
\end{remark}
The following proposition states that every \tpfa{} \tme{} induces a \tRaM{} on
its \tcor{}.
\begin{proposition}\label{prop:ram_on_core}
	Let $\dom \in \bor{\Rn}$ be bounded and $\me \in \bawl{\dom}$. 
	Then there exists a Radon measure $\sme$ supported on $\cor{\me}\subset \cl{\dom}$ such that for every $\sfun \in \Cefun{\dom}$
	\begin{equation*}
		\I{\dom}{\sfun}{\me} = \I{\cor{\me}}{\sfun}{\sme} \,.
	\end{equation*}
\end{proposition}
\begin{proof}
	See Proposition 5.24 in \cite[p. 36]{schonherr_pure_2017}.	
\end{proof}
\begin{remark}
	In the setting of the proposition above, $\sme$ is said to be a
	\textbf{representation of $\me$} on $\cor{\me}$.
\end{remark}
The next proposition gives a partial inverse to the statement of the
proposition above. In particular, any \tRaM{} can be extended to a \tme{} on
all of its domain.
\begin{proposition}\label{prop:smearing}
	Let $\dom \in \bor{\Rn}$ be bounded and $\ferm \subset \cl{\dom}$ be
	closed such
	that for every $\eR \in \ferm$ and every $\delta > 0$
	\begin{equation*}
		\lem(\ball{\eR}{\delta} \cap \dom) > 0 \,.
	\end{equation*}
	Furthermore, let $\sme$ be a Radon measure on $\ferm$. Then there
	exists $\me \in \bawl{\dom}$ such that for every $\sfun \in
	\Cefun{\dom}$
	\begin{equation*}
		\I{\dom}{\sfun}{\me} = \I{\ferm}{\sfun}{\sme} \,.
	\end{equation*}
	In particular, 
	\begin{equation*}
		\cor{\me} \subset \ferm 
	\end{equation*}
	and
	\begin{equation*}
		\tova{\me}(\dom) = \tova{\sme}(\ferm) \,.
	\end{equation*}
\end{proposition}
\begin{remark}
	The conditions of the statement are satisfied if, for example,  $\ferm \subset \mbd{\dom} \cup \mint{\dom}$.	
\end{remark}
\begin{proof}
	See Proposition 5.26 in \cite[p. 37]{schonherr_pure_2017}.	
\end{proof}
The \tme{} from the preceding proposition is \tpfa{} if the \tRaM{} is singular
with respect to Lebesgue measure.
\begin{corollary}
	Let $\dom \in \bor{\Rn}$ be bounded and $\ferm \subset \cl{\dom}$ be
	closed such
	that for every $\eR \in \ferm$ and $\delta >0$
	\begin{equation*}
		\lem(\ball{\eR}{\delta} \cap \dom) >0
	\end{equation*}
	and 
	\begin{equation*}
		\lem(\ferm \cap \dom) = 0\,.
	\end{equation*}
	Furthermore, let $\sme$ be a \tRaM{} on $\ferm$.

	Then there exists $\me \in
	\bawl{\dom}$ such that for all $\sfun \in \Ccfun{\dom}$ 
	\begin{equation*}
		\I{\dom}{\sfun}{\me} = \I{\ferm}{\sfun}{\sme}\,.
	\end{equation*}
	Furthermore, 
	\begin{equation*}
		\tova{\me}(\dom) = \tova{\sme}(\ferm)
	\end{equation*}
	and $\me$ is \tpfa{}.
\end{corollary}
\begin{proof}
	The preceding proposition and Proposition \ref{prop:corenul_pfa} yield
	the statement.
\end{proof}
The following example presents another way to construct a density at zero.
\begin{example}
	Let $\dom\in \bor{\Rn}$ be bounded and $\eR \in \cl{\dom}$ such that for every
	$\delta >0$ 
	\begin{equation*}
		\lem(\ball{\eR}{\delta} \cap \dom) > 0 \,.
	\end{equation*}
	Then there exists a \tpfa{} $\me \in \bawl{\dom}$ such that for every
	$\sfun \in \Cefun{\dom}$
	\begin{equation*}
		\I{\dom}{\sfun}{\me} = \sfun(\eR) \,.
	\end{equation*}
\end{example}
The next example shows an extension for $\ham^{n-1}$.
\begin{example}
	Let $\dom \in \bor{\Rn}$ be open, bounded and have smooth boundary. Then $\lem(\bd{\dom}) =
	0$ and $\ferm= \bd{\dom}$ satisfies the assumptions of
	Proposition \ref{prop:smearing}.
	Hence, there exists $\me \in \bawl{\dom}$ such that for all $\sfun \in
	\Cefun{\dom}$
	\begin{equation*}
		\I{\bd{\dom}}{\sfun}{\ham^{n-1}} = \I{\dom}{\sfun}{\me} \,.
	\end{equation*}
\end{example}
The following example shows, that the surface part of a Gauß formula can be
expressed as an integral with respect to a \tpfa{} measure. In Section
\ref{sec:nt_dmi} this is extended to vector fields having divergence
measure.
\begin{example}
	Let $\dom \in \bor{\Rn}$ be a bounded set with smooth boundary. Then $\ferm = \bd{\dom} \subset
	\cl{\dom}$ is a closed set and for every $k \in \N$ such that $1 \leq
	k \leq n$
	\begin{equation*}
		\bvnu{k} \cdot \reme{\ham^{n-1}}{\bd{\dom}}
	\end{equation*}
	is a \tRaM{} on $\ferm$. By Proposition \ref{prop:smearing} there
	exists $\me_k \in \bawl{\dom}$ such that for every $\sfun \in
	\Cefun{\dom}$
	\begin{equation*}
		\I{\bd{\dom}}{\sfun \cdot \bvnu{k}}{\ham^{n-1}} =
		\I{\dom}{\sfun}{\me_k} = \sI{\bd{\dom}}{\sfun}{\me_k} 
	\end{equation*}
	and 
	\begin{equation*}
		\cor{\me_k} \subset \bd{\dom} \,.
	\end{equation*}
	Hence, there exists $\me \in \bawln{\dom}$ such that for all $\sfun \in
	\ConeN{\cl{\dom}}$
	\begin{equation*}
		\sI{\bd{\dom}}{\sfun}{\me} = \I{\dom}{\sfun}{\me} = \I{\bd{\dom}}{\sfun
		\cdot \normal{}}{\ham^{n-1}} =
			\I{\dom}{\divv{\sfun}}{\lem} \,,
	\end{equation*}
	where the Gauß formula for sets with finite perimeter from Evans \cite[p. 209]{Evans1992} was used.
	Furthermore,
	\begin{equation*}
		\cor{\me} \subset \bd{\dom}
	\end{equation*}
	and $\me$ is \tpfa{} by
	Proposition \ref{prop:corenul_pfa}.
\end{example}
\section{Bounded Vector Fields having Divergence Measure}\label{sec:nt_dmi}

The following lemma enables the use of the characterisation of the dual of
$L^\infty$ in the following theorem. The key point of this statement is that
the dual of a product space is essentially the product of the dual spaces.

Nevertheless, a self-contained proof is given.
\begin{proposition}\label{prop:dual_libln}
	Let $\dom \subset \Rn$. The dual space of
	\begin{equation*}
		\libln
	\end{equation*}
	equipped with the norm
	\begin{equation*}
		\norm{}{\funv} := \sup \limits_{\substack{k \in \N\\ 1 \leq k
			\leq n}} \normi{\funv_k} \mspace{6mu} \text{ for }
			\mspace{6mu}\funv \in \libln
	\end{equation*}
	is the space
	\begin{equation*}
		\bawln{\dom}
	\end{equation*}
	equipped with the norm
	\begin{equation*}
		\norm{}{\norme} = \sum \limits_{k=1}^n |\norme_k|(\dom) \text{ for } \norme \in \bawln{\dom} \,.
	\end{equation*}
\end{proposition}
\begin{proof}
	Let $\norme \in \bawln{\dom}$. Then $\de : \libln \to \R$ defined by
	\begin{equation*}
		\df{\de}{\funv} = \sum \limits_{k=1}^n \I{\dom}{\funv_k}{\norme_k}	
	\end{equation*}
	for $\funv \in \libln$ is obviously a linear functional on $\libln$.
	Furthermore
	\begin{equation*}
		\left |\df{\de}{\funv}\right| \leq \sum \limits_{k=1}^n \normi{\funv_k}\tova{\norme_k}(\dom)\leq \norm{}{\funv}\norm{}{\norme}
	\end{equation*}
	for $\funv \in \libln$, where the norms are defined as in the statement of the proposition.

	Now let $\de \in \dual{\libln}$. Then for every $k \in \N, 1 \leq k \leq n$
	\begin{equation*}
		\de_k : \libl \to \R : \fun \mapsto \df{\de}{\fun e_k}
	\end{equation*}
	is a continuous linear functional on $\libl$. By Proposition
	\ref{prop:dual_liss_pre} there exist $\norme_k \in \bawl{\dom}$ such that
	\begin{equation*}
		\df{\de_k}{\fun} = \I{\dom}{\fun}{\norme_k}
	\end{equation*}
	for every $\fun \in \libl$. Hence for every $\funv \in \libln$
	\begin{equation*}
		\df{\de}{\funv} = \sum \limits_{k=1}^n \df{\de_k}{\funv_k} = \sum \limits_{k=1}^n \I{\dom}{\funv_k}{\norme_k} = \I{\dom}{\funv}{\norme} \,,
	\end{equation*}
	where $\norme = (\norme_1, ..., \norme_n)$.

	Now let $\norme \in \bawln{\dom}$. Then
	\begin{align*}
		\sup\limits_{\substack{\funv\in \libln\\\norm{}{\funv}\leq 1}} \left |\I{\dom}{\funv}{\norme} \right |& = \sup\limits_{\substack{\funv\in \libln\\\norm{}{\funv}\leq 1}} \I{\dom}{\funv}{\norme} \\
		&= \sup \limits_{\substack{\funv\in \libln\\\norm{}{\funv}\leq 1}} \sum \limits_{k=1}^n \I{\dom}{\funv_k}{\norme_k} \\
		& = \sum \limits_{k=1}^n \sup \limits_{\substack{\funv_k\in \libl\\ \normi{\funv_k} \leq 1}} \I{\dom}{\funv_k}{\norme_k} \\
		& = \sum \limits_{k=1}^n \tova{\norme_k}(\dom)= \norm{}{\norme}\,.
	\end{align*}
	This finishes the proof.
\end{proof}
The proof of the upcoming Gauß Theorem relies on the following notion of
approximation of the domain $\dom$. It turns out that this is not only a
technical necessity but gives the obtained Gauß formulas a more flexible shape.
\begin{definition}
	Let $\encl{\dom} \subset \Rn$ be open and $\dom \in \bor{\encl{\dom}}$
	with $\distf{\bd{\encl{\dom}}}(\dom) > 0$. 
	
	A sequence $\{\nap_k\}_{k\in\N}\subset
	\skpdr{1}{\infty}{\encl{\dom}}{[0,1]}$ of \tLcont{} real functions with
	compact support in $\encl{\dom}$ is called \textbf{\tnormapprox{} for
		$\ind{\dom}$}\index{good
		app@\tnormapprox{}}\nomenclature[f]{$\nap_k$}{\tnormapprox{} of
	the characteristic function} with limit function $\nap$, if 
	\begin{enumerate}
		\item 	
			\begin{equation*}
				\lim\limits_{k \to \infty} \nap_k(\eR) =:
				\nap(x) \mspace{6mu}\text{ exists $\ham^{n-1}$-\tmae{} on }
				\encl{\dom}
			\end{equation*}
		\item
			\begin{equation*}
				\nap = 1 \mspace{12mu}\text{$\ham^{n-1}$-\tmae{} on } \Int{\dom}
			\end{equation*}
		\item 
			\begin{equation*}
				\nap = 0
				\mspace{12mu}\text{$\ham^{n-1}$-\tmae{} on }
				\left(\cl{\dom}\right)^c
			\end{equation*}
		\item 
			\begin{equation*}\sup\limits_{k \in \N}  \norm{1}{\Deriv{\nap_k}} <
			\infty\,.
			\end{equation*}
	\end{enumerate}
\end{definition}
A necessary condition for $\dom$ to allow a \tnormapprox{} is given in the next
proposition.
\begin{proposition}
	Let $\encl{\dom}\subset \Rn$ be open and $\dom \in \bor{\encl{\dom}}$
	be bounded such that $\dist{\dom}{\bd{\encl{\dom}}} >0$. If there is a
	\tnormapprox{} for $\ind{\dom}$  with
	$\norm{1}{\nap - \ind{\dom}} = 0$. Then $\dom$ is a \tbvs{}.
\end{proposition}
\begin{proof}
	Let $\{\nap_k\}_{k\in \N}$ be a \tnormapprox{} for $\ind{\dom}$. Since $\dom$ is bounded and every $\ham^{n-1}$-\tnulset{} is a $\lem$-\tnulset{},
	\begin{equation*}
		\nap_k \xrightarrow{L^1} \ind{\dom} \,.
	\end{equation*}
	Since the \ttova{} is \tlsc{},
	\begin{equation*}
		\tova{\Deriv{\ind{\dom}}}(\encl{\dom}) \leq \liminf\limits_{k
		\to \infty} \norm{1}{\Deriv{\nap_k}} < \infty \,.
	\end{equation*}
	This proves the statement.
\end{proof}
\begin{remark}
	In Example \ref{ex:can_norme}, it is shown that every \tbvs{}
	allows a \tnormapprox{}.
\end{remark}
Now, the Gauß Theorem can be proved using \tnormapprox{}s and the
characterisation of the dual of $\liblnd{\encl{\dom}}$.
\begin{theorem}\label{thm:ggt_me}
	Let $\encl{\dom} \subset \Rn$ be open, $\dom \in \bor{\encl{\dom}}$ be
	a bounded \tbvs{} such that $\dist{\dom}{\bd{\encl{\dom}}} > 0$. Furthermore,
	let $\{\nap_k\}_{k\in \N}$ be a \tnormapprox{} with limit $\nap$. Then
	there exists $\norme \in \bawln{\encl{\dom}}$ such that for every $k\in
	\N$, $1\leq k \leq n$
	\begin{equation*}
		\cor{\norme_k} \subset \bd{\dom} \,.
	\end{equation*}
	and the \textit{Gauß formula}\index{Gauß formula}
	\begin{equation}\label{eq:ggt_me}
		\divv{\funv}\left(\Int{\dom}\right) +
		\I{\bd{\dom}}{\nap}{\divv{\funv}} = \sI{\bd{\dom}}{\funv}{\norme} 
	\end{equation}
	holds for every $\funv \in \dmid{\encl{\dom}}$.
	The \tme{} $\norme$ is minimal in the norm, i.e. if $\norme' \in \bawln{\encl{\dom}}$ satisfies \eqref{eq:ggt_me} for every $\funv \in \dmid{\encl{\dom}}$, then
	\begin{equation*}
		\norm{}{\norme} \leq \norm{}{\norme'} \,.
	\end{equation*}
	In addition, for every \tbvs{} $\bvs\in \bor{\encl{\dom}}$
	\begin{equation*}
		\norme(\bvs) = - \lim \limits_{k \to \infty} \I{\bvs\cap
		\supp{\Deriv{\nap_k}}}{\Deriv{\nap_k}}{\lem} = -
		\I{\rbd{\bvs}\cap \cl{\dom}}{\nap \cdot \normal{\bvs}}{\ham^{n-1}}
	\end{equation*}
\end{theorem}
The preceding new Gauß Theorem sets itself apart from the literature by
introducing normal measures. In the literature, Gauß formulas for sets of
finite perimeter and essentially bounded vector field can be found in the form
of functionals on a function space on the boundary (cf.
\cite[p. 448]{silhavy_divergence_2009}) or as functions on the boundary which are
obtained by mollification (cf. \cite[p. 262f]{chen_gauss-green_2009}). The
approach chosen here enables a clean separation of geometry and vector field.
In plus, it yields the existence of a normal measure which is defined on all
Borel subsets.
\begin{definition}
	$\norme\in \bawln{\encl{\dom}}$ satisfying \eqref{eq:ggt_me} for all
	$\funv \in \dmid{\encl{\dom}}$ with some
	limit $\nap$ of a \tnormapprox{} of $\ind{\dom}$ that is minimal in the
	sense of Theorem \ref{thm:ggt_me} is called
	\textbf{\tggtme}\index{normal measure@\tggtme{}} of $\dom$ related to
	$\nap$.
\end{definition}
Now, the proof of Theorem \ref{thm:ggt_me} is given.
\begin{proof}
	Let $\dom \in \bor{\encl{\dom}}$ be a bounded set of finite perimeter
	with $\dist{\dom}{\bd{\encl{\dom}}} > 0$ and let $\{\nap_k\}_{k
	\in \N} \subset \skpdr{1}{\infty}{\encl{\dom}}{[0,1]}$ be an associated
	\tnormapprox{} with limit function $\nap$.

	Now, let $\funv \in \dmid{\encl{\dom}}$. 
	Then by the Dominated Convergence Theorem (cf. \cite[p. 20]{Evans1992})
	\begin{equation*}
		\I{\encl{\dom}}{\nap_k}{\divv{\funv}} \xrightarrow{k \to
	\infty} \divv{\funv}(\Int{\dom}) + \I{\bd{\dom}}{\nap}{\divv{\funv}} \,.	
	\end{equation*}
	Note that $\divv{\funv} \wac \ham^{n-1}$ (cf.
	\cite[p. 1014]{chen_structure_2011}).
	On the other hand, 
	\begin{equation*}
		\funv \cdot \nap_k \in \dmid{\encl{\dom}}
	\end{equation*}
	by \cite[p. 101]{chen_divergencemeasure_1999}. Furthermore, $\funv \cdot
	\nap_k$ is compactly supported in $\encl{\dom}$. Thus by
	\cite[p. 252]{chen_divergence-measure_2005} for every $k \in \N$
	\begin{equation*}
		\I{\encl{dom}}{\nap_k}{\divv{\funv}} = -
		\I{\encl{\dom}}{\Deriv{\nap_k} \cdot \funv}{\lem} \,.
	\end{equation*}
		Hence for every $k \in \N$
	\begin{equation*}
		\left | \I{\encl{\dom}}{\nap_k}{\divv{\funv}} \right | \leq
		\norm{\lp^1}{\Deriv{\nap_k}} \norm{\infty, \supp{\nap_k}}{\funv} \,.
	\end{equation*}
	This implies
	\begin{equation*}
		\left|  \divv{\funv}(\Int{\dom}) +
		\I{\bd{\dom}}{\nap}{\divv{\funv}} \right | \leq \limsup
		\limits_{k \to \infty} \norm{\lp^1}{\Deriv{\nap_k}}
		\norm{\infty,\supp{\nap_k}}{\funv} \leq \sup\limits_{k\in \N}
		\norm{1}{\Deriv{\nap_k}} \norm{\infty}{\funv} \,.
	\end{equation*}
	Hence
	\begin{equation*}
		\de_0 : \dmid{\encl{\dom}} \to \R : \funv \mapsto
	\divv{\funv}(\Int{\dom}) + \I{\bd{\dom}}{\nap}{\divv{\funv}}  
	\end{equation*}
	is a continuous linear functional on a subspace of
	$\liblnd{\encl{\dom}}$. By the Hahn-Banach Theorem
	\cite[p. 63]{dunford_linear_1988}
	there exists a continuous linear extension $\de$ of $\de_0$ to all of
	$\liblnd{\encl{\dom}}$ such that $\norm{}{\de} = \norm{}{\de_0}$. In
	particular, this extension is minimal in the norm.
	By Proposition \ref{prop:dual_libln} there exists a $\norme \in \bawln{\encl{\dom}}$ such that for all $\funv \in \dmi$
	\begin{equation*}
		\divv{\funv}(\Int{\dom}) + \I{\bd{\dom}}{\nap}{\divv{\funv}} = \I{\encl{\dom}}{\funv}{\norme} \,.
	\end{equation*}
	and $\norm{}{\norme}=\norm{}{\de_0}$.
	Furthermore
	\begin{equation*}
		\sum \limits_{k=1}^n \tova{\norme_k}(\dom) = \norm{}{\de_0} = \norm{}{\de} \,.
	\end{equation*}
	Note that by the Coarea Formula (cf. \cite[p. 112]{Evans1992}), for
	\tmae{} $0< \delta <
	\dist{\dom}{\bd{\encl{\dom}}}$ the neighbourhood $\dnhd{\dom}{\delta}$
	is a \tbvs{}. By \cite[p. 101]{chen_divergencemeasure_1999}
	\begin{equation*}
		\funv \cdot \ind{\dnhd{\dom}{\delta}} \in \dmid{\encl{\dom}}
		\,.
	\end{equation*}
	But $\funv \cdot \ind{\dnhd{\dom}{\delta}}$ and $\funv$ agree on a
	neighbourhood of $\dom$, whence
	\begin{equation*}
		\divv{(\funv\cdot\ind{\dnhd{\dom}{\delta}})} (\Int{\dom}) +
		\I{\bd{\dom}}{\nap}{\divv{(\funv \cdot
	\ind{\dnhd{\dom}{\delta}})}} = \divv{\funv}(\Int{\dom}) +
\I{\bd{\dom}}{\nap}{\divv{\funv}}\,.
	\end{equation*}
	whence
	\begin{equation*}
		\I{\encl{\dom}\setminus\dnhd{\dom}{\delta}}{\funv}{\norme} = 0 
	\end{equation*}
	for every $\funv \in \dmid{\encl{\dom}}$ and \tmaet{} $0< \delta < \dist{\dom}{\bd{\encl{\dom}}}$.
	Thus for \tmaet{} such $\delta>0$ and $\funv\in \dmid{\encl{\dom}}$
	\begin{equation*}
		\df{\de}{\funv} = \I{\dom}{\funv}{\reme{\norme}{\dnhd{\dom}{\delta}}} \,.
	\end{equation*}
	This implies 
	\begin{equation*}
		\norm{}{\de_0} \leq \norm{}{\reme{\norme}{\dnhd{\dom}{\delta}}} \leq \norm{}{\norme} = \norm{}{\de_0}
	\end{equation*}
	and thus
	\begin{equation*}
		\sum \limits_{k=1}^n \tova{\reme{\norme_k}{(\encl{\dom}
		\setminus \dnhd{\dom}{\delta})}}(\dom) = \sum \limits_{k=1}^n
		\tova{\norme_k}(\encl{\dom}) - \tova{\reme{\norme_k}{\dnhd{\dom}{\delta}}}(\dom) = \norm{}{\de_0} -\norm{}{\de_0} = 0 
	\end{equation*}
	whence
	\begin{equation*}
		\tova{\norme_k}\left ( \encl{\dom} \setminus \dnhd{\dom}{\delta} \right ) = 0 
	\end{equation*}
	for \tmaet{} $0<\delta < \dist{\dom}{\bd{\encl{\dom}}}$.
	Since $\tova{\norme_k}$ is monotone, the statement follows for all
	$0<\delta < \dist{\dom}{\bd{\encl{\dom}}}$.

	Note that by the Coarea formula (cf. \cite[p. 112]{Evans1992})
	$\idnhd{\dom}{\delta}$ is a \tbvs{} for \tmaet{} $\delta > 0$. By \cite[p. 101]{chen_divergencemeasure_1999}
	\begin{equation*}
		\funv \cdot \ind{\idnhd{\dom}{\delta}} \in \dmid{\encl{\dom}} \,.
	\end{equation*}
	
	Since $\dom$ is bounded, $\funv \cdot \ind{\idnhd{\dom}{\delta}}$ is
	compactly supported in $\Int{\dom}$ and thus by
	\cite[p. 252]{chen_divergence-measure_2005}
	\begin{equation*}
		\divv{\funv \cdot \ind{\idnhd{\dom}{\delta}}}(\Int{\dom}) +
	\I{\bd{\dom}}{\nap}{\divv{\left(\funv \cdot \ind{\idnhd{\dom}{\delta}}\right)}} = 0 \,.
	\end{equation*}
	This implies
	\begin{equation*}
		\I{\dom}{\funv}{\reme{\norme}{\idnhd{\dom}{\delta}}} = 0 
	\end{equation*}
	for every such $\delta > 0$ and $\funv \in \dmid{\encl{\dom}}$.
	Hence
	\begin{equation*}
		\norm{}{\de_0} \leq \norm{}{\reme{\norme}{\encl{\dom}\setminus\idnhd{\dom}{\delta}}} \leq \norm{}{\norme} = \norm{}{\de_0} \,.
	\end{equation*}
	and analogously to the reasoning for $\dnhd{\dom}{\delta}$ one deduces
	\begin{equation*}
		\tova{\norme_k}\left(\idnhd{\dom}{\delta}\right) = 0 
	\end{equation*}
	for every $k \in \N, 1 \leq k \leq n$.

	Now for every $\delta > 0, \delta < \dist{\dom}{\bd{\encl{\dom}}}$ and every $k \in \N, 1 \leq k \leq n$
	\begin{equation*}
		\tova{\norme_k}\left( \left(\encl{\dom} \setminus \dnhd{\dom}{\delta} \right ) \cup \idnhd{\dom}{\delta} \right ) = 0 \,.	
	\end{equation*}
	This implies
	\begin{equation*}
		\cor{\norme_k} \subset \bd{\dom}
	\end{equation*}
	for every $k \in \N, 1 \leq k \leq n$. This establishes Equation
	\eqref{eq:ggt_me}.

	Now, let $\bvs \in \bor{\encl{\dom}}$ be a \tbvs{}.
	Then for $k \in \N, 1\leq k \leq n$
	\begin{equation*}
		e_k\cdot \ind{\bvs} \in \dmid{\encl{\dom}} \,.
	\end{equation*}
	Note that
	\begin{equation*}
		\divv{(e_k\cdot \ind{\bvs})} = \pD{\ind{\bvs}}{k} = -
		\acme{(\normal{\bvs})_k}{\reme{\ham^{n-1}}{\rbd{\bvs}}} \,.
	\end{equation*}
	The established Gauß formula yields
	\begin{align*}
		\I{\encl{\dom}}{\ind{\bvs}}{\norme_k} & = \I{\encl{\dom}}{e_k \cdot \ind{\bvs}}{\norme} \\
						      &	= \divv{(e_k \cdot \ind{\bvs})}(\Int{\dom}) + \I{\bd{\dom}}{\nap}{\divv{(e_k\cdot \ind{\bvs})}}\\
						      & = \I{\encl{\dom}}{\nap}{\divv{(e_k\cdot\ind{\bvs})}} \\
						      & = \lim \limits_{l \to \infty} \I{\encl{\dom}}{\nap_l}{\divv{(e_k\cdot\ind{\bvs})}}\\
						      & = \lim \limits_{l \to \infty} - \I{\bvs}{e_k \cdot\Deriv{\nap_l}}{\lem}
	\end{align*}
	Since $k \in \N, 1\leq k \leq n$ was arbitrary
	\begin{equation*}
		\norme(\bvs) = -\lim \limits_{k \to \infty}\I{\bvs\cap
		\supp{\Deriv{\nap_k}}}{\Deriv{\nap_k}}{\lem} \,.
	\end{equation*}
	On the other hand, for every \tbvs{} $\bvs\in \bor{\encl{\dom}}$ and
	for $k \in \N$, $1 \leq k \leq n$
	\begin{align*}
		\norme_k(\bvs) & = \lim \limits_{l \to
		\infty}\I{\encl{\dom}}{\nap_l}{\divv{(e_k \cdot \ind{\bvs})}} \\
										       &
		= \lim \limits_{l \to \infty}\I{\encl{\dom}}{\nap_l}{\pD{(\ind{\bvs})}{k}} \\
										       &
		=  - \lim \limits_{l \to \infty} \I{\rbd{\bvs}}{\nap_l \cdot
		(\normal{\bvs})_k}{\ham^{n-1}} \\
										       &
		= - \I{\rbd{\bvs}}{\nap \cdot (\normal{\bvs})_k}{\ham^{n-1}} \,.
	\end{align*}
	Hence
	\begin{equation*}
		\norme(\bvs) = - \I{\rbd{\bvs}\cap \cl{\dom}}{\nap \cdot \normal{\bvs}}{\ham^{n-1}} \,.	
	\end{equation*}
\end{proof}
Given a \tnormapprox{} of $\ind{\dom}$, normal measures are uniquely
defined on sets of finite perimeter.
\begin{proposition}
	Let $\encl{\dom}\subset \Rn$ be open and $\dom \in \bor{\encl{\dom}}$
	be a bounded \tbvs{} such that $\dist{\dom}{\bd{\encl{\dom}}}>0$.
	Let $\{\nap_k\}_{k \in \N}$ be a \tnormapprox{} with limit $\nap$.
	Let $\norme \in \bawln{\encl{\dom}}$ be an associated \tggtme{}. Then
	for every \tbvs{} $\bvs \in \bor{\encl{\dom}}$ there exists a Lebesgue \tnulset{} $N \subset \R$
	\begin{equation*}
		\norme(\bvs) = \I{\rbd{\bvs} \cap \bd{\dom}}{-\nap
			\normal{\bvs}}{\ham^{n-1}} + \lim \limits_{\substack{\delta \downarrow 0\\\delta \notin N}}\I{\mint{\bvs} \cap \rbd{\idnhd{\dom}{\delta}}}{\normal{\idnhd{\dom}{\delta}}}{\ham^{n-1}}
	\,.
	\end{equation*}
\end{proposition}
\begin{proof} Note that $\bvs$ and $\mint{\bvs}$ only differ by a
	$\lem$-\tnulset{} (cf. \cite[p. 43]{Evans1992}). Hence $\mint{\bvs}$ is also a set of finite
	perimeter. W.l.o.g. $\bvs = \mint{\bvs}$.
	The Coarea formula (cf. \cite[p. 112]{Evans1992}) implies that for
	\tmae{} $\delta >0$ the set
	\begin{equation*}
		\dnhd{\dom}{\delta}\setminus \idnhd{\dom}{\delta}
	\end{equation*}
	has finite perimeter. Then
	$\mint{(\dnhd{\dom}{\delta}\setminus\idnhd{\dom}{\delta})}$ is also a
	\tbvs{}. By \cite[p. 5]{gurtin_geometric_1986}, 
	\begin{equation*}
		\bvs \cap \mint{(\dnhd{\dom}{\delta}\setminus
		\idnhd{\dom}{\delta})}
	\end{equation*}
	is also a set of finite perimeter. Note that the Coarea Formula also
	implies that for \tmae{} $\delta >0$
	\begin{equation*}
		\ham^{n-1}(\mbd{\bvs} \cap \bd{(\dnhd{\dom}{\delta}\setminus
		\idnhd{\dom}{\delta})}) = 0 \,.
	\end{equation*}
	Using this and \cite[p. 199]{degiovanni_cauchy_1999}
	\begin{equation*}
		\mbd{(\bvs \cap
		\mint{(\dnhd{\dom}{\delta}\setminus\idnhd{\dom}{\delta})})}
	\end{equation*}
	differs from 
	\begin{equation*}
		(\mbd{\bvs} \cap \mint{(\dnhd{\dom}{\delta}\setminus
			\idnhd{\dom}{\delta})}) \cup (\bvs \cap
			\mbd{(\dnhd{\dom}{\delta}\setminus
			\idnhd{\dom}{\delta})})
	\end{equation*}
	only by a $\ham^{n-1}$-\tnulset{}. Since $\dnhd{\dom}{\delta}\setminus
	\idnhd{\dom}{\delta}$ has density $1$ at points of its measure
	theoretic interior and the measure theoretic normal is characterised by
	the halfspace it generates (cf. \cite[p. 203]{Evans1992}), one sees that
	\begin{equation*}
		\normal{\bvs \cap
			(\dnhd{\dom}{\delta}\setminus\idnhd{\dom}{\delta})} =
			\normal{\bvs} \mspace{6mu} \text{ on } \mspace{6mu} \rbd{\bvs} \cap
			\mint{(\dnhd{\dom}{\delta}\setminus
				\idnhd{\dom}{\delta})} \cap \rbd{(\bvs \cap
			(\dom \setminus \idnhd{\dom}{\delta}))}\,.
	\end{equation*}

	Theorem \ref{thm:ggt_me} states $\cor{\norme}\subset \bd{\dom}$, thus
	\begin{equation*}
		\norme(\bvs) = \norme(\bvs \cap
		\mint{(\dnhd{\dom}{\delta}\setminus \idnhd{\dom}{\delta})}) =
		- \I{\rbd{(\bvs \cap \mint{(\dnhd{\dom}{\delta}\setminus
			\idnhd{\dom}{\delta})})} \cap
				\cl{\dom}}{\nap\normal{(\bvs \cap
					\mint{(\dnhd{\dom}{\delta}\setminus
			\idnhd{\dom}{\delta})})}}{\ham^{n-1}} \,.
	\end{equation*}
	for \tmae{} $\delta >0$. The integral on the right hand side is for
	\tmae{} $\delta >0$ equal to 
	\begin{equation*}
		\I{\rbd{\bvs} \cap \cl{\dom} \cap 
		\mint{(\dnhd{\dom}{\delta}\setminus\idnhd{\dom}{\delta})}}{-\nap\normal{\bvs}}{\ham^{n-1}}
		+ \I{\bvs \cap \cl{\dom}
		\cap \rbd{(\idnhd{\dom}{\delta})}}{\nap\normal{\idnhd{\dom}{\delta}}}{\ham^{n-1}}
		\,.
	\end{equation*}
	Noting that
	\begin{equation*}
		\I{\rbd{\bvs} \cap \cl{\dom} \cap
		\als}{-\nap\normal{\bvs}}{\ham^{n-1}}
	\end{equation*}
	defines a \tsme{} in $\als$ and using continuity from above yields
	\begin{equation*}
		\lim \limits_{\delta \downarrow 0} \I{\rbd{\bvs} \cap \cl{\dom}
			\cap \mint{(\dnhd{\dom}{\delta}\setminus
		\idnhd{\dom}{\delta})}}{-\nap\normal{\bvs}}{\ham^{n-1}} = \I{\rbd{\bvs} \cap \bd{\dom}}{-\nap \normal{\bvs}}{\ham^{n-1}}
	\end{equation*}
	On the other hand $\rbd{(\idnhd{\dom}{\delta})} \subset \cl{\dom}$.
	Hence
	\begin{equation*}
		\rbd{(\idnhd{\dom}{\delta})} \cap \cl{\dom} \cap \bvs = \bvs
		\cap \rbd{(\idnhd{\dom}{\delta})}  \,.
	\end{equation*}
	Furthermore $\nap =1$ on $\Int{\dom}$.
	This finishes the proof.
\end{proof}
The following picture illustrates the representation of a normal measure from
the preceding proposition.
\begin{figure}[H]
	\centering
	\begin{tikzpicture}[scale=1.2]
		\draw[gray] (-2,-2) rectangle (2,2);
		\draw[gray] (-0.0,0) node {$\dom$};
		\filldraw[pattern=north east lines] (1.5,-1.5) rectangle
			(2,1.5);
		\draw[dashed] (0.5,-1.5) rectangle (2,1.5);
		\draw (2.2,0.5) node {$\bvs$};
		\draw (1.75,1.6) node[rotate=-90] {$\{$};
		\draw (1.75,1.8) node {$\delta$};
		\draw[->,thick] (1.5,-0.1) to (2.5,-0.1);
		\draw (2.5,-0.3) node {$\normal{\idnhd{\dom}{\delta}}$};
		\draw[->,thick] (2,0.1) to (1.2,0.1);
		\draw (1.0,0.3) node {$-\nap\normal{\bvs}$};
		\filldraw[pattern=north east lines] (-1.5,-1.5) rectangle
			(-2,1.5);
		\draw[dashed] (-0.5,-1.5) rectangle (-3.0,1.5);
		\draw (-2.5,0.5) node {$\bvs$};
		\draw (-1.75,1.6) node[rotate=-90] {$\{$};
		\draw (-1.75,1.8) node {$\delta$};
		\draw[->,thick] (-1.5,0) to (-2.5,0);
		\draw (-2.5,-0.3) node {$\normal{\idnhd{\dom}{\delta}}$};

	\end{tikzpicture}
	\caption{Domain of influence for a normal measure and a \tbvs{} $\bvs$}\label{fig:norme_rep}
\end{figure}
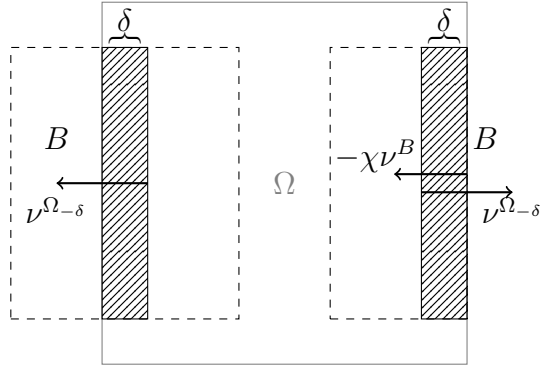
The relation of $\reme{\ham^{n-1}}{\rbd{\dom}}$ and $\tova{\norme}$ is treated
in the next proposition.
\begin{proposition}\label{prop:tova_norme}
	Let $\encl{\dom} \subset \Rn$ be open, $\dom \in \bor{\encl{\dom}}$ be
	a bounded \tbvs{} such that $\dist{\dom}{\bd{\encl{\dom}}} > 0$. Furthermore,
	let $\{\nap_k\}_{k\in \N}$ be a \tnormapprox{} with limit $\nap$ and
	let $\norme \in \bawln{\encl{\dom}}$ be the associated \tggtme{}.

	If $\norm{1}{\nap - \ind{\dom}} = 0$, then for every open set $\bals
	\subset \encl{\dom}$
	\begin{equation*}
		\tova{\norme}(\bals) \geq (\reme{\ham^{n-1}}{\rbd{\dom}})(\bals) \,.
	\end{equation*}
\end{proposition}
\begin{remark}
	Note that $\lem(\bd{\dom})=0$ implies $\norm{1}{\nap-\ind{\dom}} = 0$.
\end{remark}
\begin{proof}
	Let $\sfun \in \ConecN{\encl{\dom}}$. Then using the Gauß Theorem from
	Evans \cite[p. 209]{Evans1992}
	\begin{align*}
		\I{\encl{\dom}}{\sfun}{\norme} & =
		\I{\Int{\dom}}{\divv{\sfun}}{\lem} + \I{\bd{\dom}}{\nap
		\divv{\sfun}}{\lem} \\
		& = \I{\dom}{\divv{\sfun}}{\lem} \\
		& = \I{\rbd{\dom}}{\sfun \cdot \normal{\dom}}{\ham^{n-1}} \,.
	\end{align*}
	Hence for every open set $\bals\subset \encl{\dom}$
	\begin{align*}
		\tova{\norme}(\bals) & \geq \sup \limits_{\substack{\sfun\in
		\ConecN{\bals},\\\normi{\sfun} \leq 1}}
		\I{\encl{\dom}}{\sfun}{\norme} \\
		& \geq \sup \limits_{\substack{\sfun \in
		\ConecN{\bals},\\\norm{C}{\sfun} \leq 1}} \I{\rbd{\dom}}{\sfun
		\cdot \normal{\dom}}{\ham^{n-1}} \\
		& = \tova{\Deriv{\ind{\dom}}}(\bals) \\
		& = (\reme{\ham^{n-1}}{\rbd{\dom}})(\bals) \,.
	\end{align*}
	For the last equality, see e.g. \cite[p. 205]{Evans1992}.
	
	Since $\bals \in \bor{\encl{\dom}}$ was arbitrary, this finishes the
	proof.
\end{proof}
The following example shows that for every \tbvs{} there exists a canonical
normal measure. Hence, Theorem \ref{thm:ggt_me} is always applicable.
\begin{example}{Canonical \tggtme{}\\}\label{ex:can_norme}
	Let $\encl{\dom} \subset \Rn$ be open and $\dom \in \bor{\encl{\dom}}$
	be a bounded \tbvs{} such that $\dist{\dom}{\bd{\encl{\dom}}}>0$.
	Furthermore, let $\mol \in \Sfun{\Rn}$ be the standard \tmol{} (cf.
	\cite[p. 122]{Evans1992}). 
	Then
	\begin{equation*}
		\nap_k(\eR) := \I{\Rn}{\frac{1}{k^n} \mol\left(k (\eRR - \eR)
		\right) \ind{\dom}(\eR)}{y} \,.
	\end{equation*}
	is a \tnormapprox{} for $\ind{\dom}$.
	The limit function $\nap$ satisfies 
	\begin{equation*}
		\nap = \ind{\mint{\dom}} + \frac{1}{2}\ind{\rbd{\dom}} \text{ $\ham^{n-1}$-\tmae{}} \,.
	\end{equation*}
	See \cite[p. 175]{Ambrosio2000} for reference.
	Hence, there exists a \tggtme{} $\nu \in \bawln{\encl{\dom}}$ such that
	for every $\funv\in \dmid{\encl{\dom}}$ the following \textit{Gauß
	formula} holds
	\begin{equation*}
		\divv{\funv}(\mint{\dom}) + \frac{1}{2}\divv{\funv}(\rbd{\dom})
		= \sI{\bd{\dom}}{\funv}{\nu} \,.
	\end{equation*}
	Furthermore,
	\begin{equation*}
		\cor{\nu} \subset \bd{\dom} \,.
	\end{equation*}
\end{example}
The divergence on the regular boundary of $\dom$, weighted with $\frac{1}{2}$,
cannot be found in the literature. This is due the fact, that the majority of
the texts prohibit the vector fields under consideration from exhibiting such
concentrations. The remaining sources treat settings similar to the one of
Theorem \ref{thm:ggt_me_all} below. The weight $\frac{1}{2}$ appears plausible,
when interpreting the divergence as source strength of the field $\funv$. At
points of the regular boundary, $\dom$ geometrically resembles a half-space.
Then half of the source strength can be seen to flow into the domain and the other half flows
outwards.

The next example shows that for many closed sets of finite perimeter a more
familiar form of the Gauß Theorem can be derived.
\begin{example}{Outer \tggtme{}\\}\label{ex:outer_port}
	Let $\encl{\dom}\subset \Rn$ be open and $\dom \in \bor{\dom}$ be a
	bounded, closed \tbvs{} such that $\delta_0 := \dist{\dom}{\bd{\encl{\dom}}}>0$.
	Furthermore, let there be a sequence $\{\delta_k\}_{k \in \N}\subset
	(0,\infty)$ such
	that $\lim \limits_{k \to \infty} \delta_k = 0$ and 
	\begin{equation*}
		\sup \limits_{k \in \N} \mI{(0,\delta_k)}{\ham^{n-1}(\bd{\dnhd{\dom}{\delta}})}{\delta}
		< \infty \,.
	\end{equation*}
	This is the case if, e.g. 	
	\begin{equation*}
		\lim \limits_{\delta \downarrow 0 }
		\ham^{n-1}(\bd{\dnhd{\dom}{\delta}}) =
		\ham^{n-1}(\rbd{\dom}) \,.
	\end{equation*}
	For $\eR \in \encl{\dom}$ and $k \in \N$ set
	\begin{align*}
		\nap_k(\eR) & := \ind{\dnhd{\dom}{\delta_k}}(\eR) \left(1 -
			\frac{1}{\delta_k} \ind{\dnhd{\dom}{\delta_k}\setminus
		\dom} \dist{\dom}{\eR}\right)\\ &  =
				\max\left\{0,\min\left\{1,1-\frac{1}{\delta_k}\distf{\dom}\right\}\right\}
				\,.
	\end{align*}
	Then $\nap_k \in \skpdr{1}{\infty}{\encl{\dom}}{[0,1]}$ is \tLcont{} (cf.
	\cite[p. 47]{clarke_optimization_1984}).
	These functions are called (outer) \textbf{Portmanteau
	functions}\index{Portmanteau functions}.
	Note that by the Coarea formula for functions of bounded variation (cf. \cite[p. 185]{Evans1992})
	\begin{equation*}
		\norm{1}{\D{\nap_k}} =
		\I{(0,1)}{\ham^{n-1}(\nap_k^{-1}(\delta))}{\delta}=
		\mI{(0,\delta_k)}{\ham^{n-1}(\bd{\dnhd{\dom}{\delta}})}{\delta}
		\,.
	\end{equation*}
	Hence, the sequence $\{\nap_{\delta_k}\}_{k\in \N}$ is a \tnormapprox{}
	for $\ind{\dom}$ and the limit function is
	\begin{equation*}
		\nap = \ind{\dom} \,.
	\end{equation*}
	Thus, there exists a \tggtme{} $\nu \in \bawln{\encl{\dom}}$ such that
	for every $\funv \in \dmid{\encl{\dom}}$ the following \textit{Gauß
	formula} holds
	\begin{equation*}
		\divv{\funv}(\dom) = \sI{\bd{\dom}}{\funv}{\nu} \,.
	\end{equation*}
\end{example}
Open \tbvs{} can be treated similarly, as the following example shows.
\begin{example}{Inner \tggtme{}\\}\label{ex:inner_port}
	Let $\encl{\dom}\subset \Rn$ be open and $\dom \subset \encl{\dom}$ be
	a bounded, open \tbvs{} such that $\dist{\dom}{\bd{\encl{\dom}}}>0$.
	Furthermore, assume there exists $\{\delta_k\}_{k \in \N} \subset
	(0,\infty)$ such that $\lim \limits_{k \to \infty} \delta_k = 0$ and
	\begin{equation*}
		\sup \limits_{k \in \N} \mI{(0,\delta_k)}
		{\ham^{n-1}(\bd{\idnhd{\dom}{\delta}})}{\delta} < \infty \,.
	\end{equation*}
	For $k \in \N$ and $\eR \in \encl{\dom}$ set
	\begin{align*}
		\nap_k(\eR) & := \ind{\idnhd{\dom}{\delta_k}}(\eR) +
		\frac{1}{\delta_k} \dist{\dom^c}{\eR}
		\ind{\dom\setminus\idnhd{\dom}{\delta_k}}\\
		& = \min\left\{1,\max\left\{0,\frac{1}{\delta_k}\distf{\dom^c}\right\}\right\}\,.
	\end{align*}
	Then $\nap_k \in \skpdr{1}{\infty}{\encl{\dom}}{[0,1]}$ is \tLcont{} (cf.
	\cite[p. 47]{clarke_optimization_1984}).
	Then as in Example \ref{ex:outer_port}, the sequence $\{\nap_{\delta_k}\}_{k\in \N}$ is a \tnormapprox{} for
	$\ind{\dom}$ and the limit function is
	\begin{equation*}
		\nap = \ind{\dom} \,.
	\end{equation*}
	These functions are called (inner) \textbf{Portmanteau functions}.
	Hence there exists a \tggtme{} $\norme\in \bawln{\encl{\dom}}$ such
	that for every vector field $\funv \in \dmid{\encl{\dom}}$ the following
	\textit{Gauß formula}\index{Gauß formula} holds
	\begin{equation*}
		\divv{\funv}(\dom) = \sI{\bd{\dom}}{\funv}{\norme} \,.
	\end{equation*}
\end{example}
The subsequent corollary illustrates the dependence of the integral with
respect to \tggtme{} on the \tnormapprox{} of $\ind{\dom}$.
\begin{corollary}
	Let $\encl{\dom}\subset \Rn$ and $\dom \in \bor{\encl{\dom}}$ be a
	bounded \tbvs{} such
	that $\dist{\dom}{\bd{\encl{\dom}}}> 0$. Then for any two
	\tnormapprox{}s $\{\nap_k^1\}_{k \in \N}, \{\nap_k^2\}_{k \in \N}
	\subset \skpdr{1}{\infty}{\encl{\dom}}{[0,1]}$ for $\ind{\dom}$, associated
	\tggtme{}s $\norme_1,\norme_2\in \bawln{\encl{\dom}}$ and any $\funv\in \dmid{\encl{\dom}}$
	\begin{equation*}
		\sI{\bd{\dom}}{\funv}{\norme_1 - \norme_2} =
		\I{\bd{\dom}}{\nap_1 - \nap_2}{\divv{\funv}}
	\end{equation*}
	where $\nap_1$ and $\nap_2$ are the limit functions for
	$\{\nap_k^1\}_{k \in N}$ and $\{\nap_k^2\}_{k \in \N}$ respectively.
\end{corollary}
\begin{remark}
	In particular, if $\tova{\divv{\funv}}(\bd{\dom})=0$, 
	\begin{equation*}
		\sI{\bd{\dom}}{\funv}{\norme}
	\end{equation*}
	is independent of the choice of the \tnormapprox{}.
\end{remark}
Since $\norme{}$ is a bounded \tme{}, all essentially bounded vector fields
$\funv$ are integrable with respect to this \tme{}. This leads to the question,
whether $\funv \in \lpbmd{1}{\norme}{\encl{\dom}}$ for unbounded vector fields.
The next example answers this negatively. The function is similar to the one in \cite[p. 100]{chen_divergencemeasure_1999}.
\begin{example}\label{ex:tang_unbound}
	Let $\encl{\dom}:=(0,1)^2 \subset \R^2$ and $\dom := \{(x,y)\in R^2 | x \leq y \}\cap \ball{\frac{1}{2}(1,1)}{\frac{1}{4}}$. Furthermore let $\norme \in \bawln{\encl{\dom}}$ be a normal measure for $\dom$ and $\funv\in \dmone$ defined by
	\begin{equation*}
	\funv(\eR,\eRR) := |\eR-\eRR|^{-\frac{1}{2}}\begin{pmatrix} 1 \\ 1 \end{pmatrix}
	\end{equation*}
	for $x\neq y$.
	Then $\divv{\funv} = 0$. In order to see that, let $\Delta := \{(\eR,\eR) \in \encl{\dom} | \eR \in \R\}$, $1 > \delta > 0$ and note that for $\sfun \in \Conec{\encl{\dom}}$
	\begin{align*}
		\I{\encl{\dom}}{\funv \cdot \Deriv{\sfun}}{\lem} & = \I{\encl{\dom} \cap \dnhd{\Delta}{\delta}}{\funv \cdot \Deriv{\sfun}}{\lem} + \I{\encl{\dom} \setminus \dnhd{\Delta}{\delta}}{\funv \cdot \Deriv{\sfun}}{\lem} \\
																																								       & = \I{\encl{\dom}\cap\dnhd{\Delta}{\delta}}{\funv \cdot \Deriv{\sfun}}{\lem} - \I{\encl{\dom} \setminus \dnhd{\Delta}{\delta}}{\sfun\divv{\funv}}{\lem} - \I{\bd{(\encl{\dom} \setminus \dnhd{\Delta}{\delta})}}{\sfun \funv \cdot \normal{\dnhd{\Delta}{\delta}}}{\ham^1} 
	\end{align*}
Since $\divv{\funv} = 0$ outside of $\dnhd{\Delta}{\delta}$ and $\funv \cdot
\normal{\dnhd{\Delta}{\delta}} = \pm \begin{pmatrix} 1 \\ -1 \end{pmatrix} \cdot \funv = 0$ on $\encl{\dom}$ 
	\begin{equation*}
		\I{\encl{\dom}}{\funv \cdot \Deriv{\sfun}}{\lem} = \I{\encl{\dom} \cap \dnhd{\Delta}{\delta}}{\funv \cdot \Deriv{\sfun}}{\lem} \xrightarrow{\delta \downarrow 0} 0 \,.
	\end{equation*}
	Note that for every $\const >0$ with $\frac{1}{\const^2}>\delta >0$
	\begin{equation*}
		\dnhd{\dom}{\delta} \cap \left\{(x,y) \in \encl{\dom}\setpipe 
		|\funv(x,y)| \geq \const\right\} \supset
		\dnhd{\Delta}{\delta} \cap \dnhd{\dom}{\delta}
	\end{equation*}
	Hence for every $\funv' \in \liblnd{\encl{\dom}}$
	\begin{align*}
		\tova{\norme}& \left(\dnhd{\dom}{\delta} \cap \left\{(x,y) \in
		\encl{\dom} \setpipe |\funv(x,y) - \funv'(x,y)| \geq \eps\right\}\right) \\
		& \geq \tova{\norme}\left(\dnhd{\dom}{\delta} \cap \left\{(x,y)
\in \encl{\dom} \setpipe |\funv (x,y)| \geq \normi{\funv'}+ \eps \right\}\right) \\
				  & \geq \tova{\norme}
		\left(\dnhd{\Delta}{\delta}\cap \dnhd{\dom}{\delta}\right) \geq
		\ham^1 \left(\Delta \cap
			\ball{\frac{1}{2}(1,1)}{\frac{1}{4}}\right) =
			\frac{1}{2} > 0
	\end{align*}
for every $0 < \delta < \frac{1}{(\eps + \normi{\funv'})^2}$.

Hence, there is no sequence $\{\funv_k\}_{k \in \N} \subset
\liblnd{\encl{\dom}}$ converging \tconvim{} to $\funv$. In particular, $\funv$
cannot be approximated in measure by \tsimf{} functions.
\end{example}
\begin{remark}
	The preceding example indeed works for $\encl{\dom} = (0,1)^2$ and every $\funv \in \dmone$ such that for some $\sfun : \R \to \R$ satisfying
	\begin{enumerate}
		\item $\sfun$ is continuously differentiable on $\R\setminus\{0\}$
		\item $\lim \limits_{\eR \to 0} \sfun(\eR) = \infty$
		\item $g : \encl{\dom} \to \R : (\eR,\eRR) \mapsto \sfun(\eR-\eRR)$ is integrable on $\encl{\dom}$
	\end{enumerate}
	it holds
	\begin{equation*}
		\funv = g \cdot \begin{pmatrix}
			1 \\ 1
		\end{pmatrix} \,.
	\end{equation*}
	The essential point is that $\funv$ is tangential to the curve where it is unbounded.
	Hence, there are many vector fields which cannot even be approximated
	in measure.
\end{remark}
The following example gives a vector field that only blows up at one point and
still is not integrable with respect to \tggtme{}. The function is the same as
in \cite[p. 403]{chen_theory_2001}.
\begin{example}\label{ex:dirac_mass}
	Let $n=2$, $\encl{\dom}:= \ball{0}{1}\subset \R^2$ and 
	\begin{equation*}
		\dom:= \ball{0}{\frac{1}{2}} \cap \{(\eR,\eRR) \in \R^2|
		\eR,\eRR\geq 0\} \,.
	\end{equation*} 
	Furthermore, let 
	\begin{equation*}
		\funv: \encl{\dom} \to \R^2 : \eR \mapsto \frac{1}{2\pi}\frac{\eR}{|\eR|^2} \,.
	\end{equation*}
	Then $\funv \in \dmone$ and $\divv{\funv}=\drm_0$.
	Let $\{\nap_k\}_{k \in\N}$ be the canonical \tnormapprox{} from
	Example \ref{ex:can_norme}. 
	Let $\norme \in \bawln{\encl{\dom}}$ be the \tggtme{} associated with this \tnormapprox{}. Assume that $\funv \in \lpbmd{1}{\norme}{\encl{\dom}}$. Then
	\begin{equation*}
		\I{\encl{\dom}}{|\funv|}{\tova{\norme}} < \infty \,.
	\end{equation*}
	But 
	\begin{equation*}
		\I{\encl{\dom}}{|\funv|}{\tova{\norme}} \geq \frac{1}{2\pi}
		\I{\bd{\dom}}{\frac{1}{|x|}}{\ham^1} \geq \frac{1}{2\pi}\I{\left(0,\frac{1}{2}\right)}{\frac{1}{t}}{t}= \infty \,,
	\end{equation*}
	a contradiction. Hence $\funv \notin \lpbmd{1}{\norme}{\encl{\dom}}$.
\end{example}
Up to now, the Gauss Theorem was given for sets that have a positive distance
to the boundary. In order to complement this result, the following theorem
states the theorem for the whole set, in the case of $\encl{\dom} = \dom$.
\begin{theorem}\label{thm:ggt_me_all}
	Let $\dom \in \bor{\Rn}$ be a bounded open \tbvs{}. If there exists
	$\delta_0>0$ and $\const >0$ such that for \tmaet{} $\delta \in
	(0,\delta_0)$
	\begin{equation*}
		\ham^{n-1}(\bd{\idnhd{\dom}{\delta}}) \leq \const \,,
	\end{equation*}
	then there exists $\norme \in
	\bawln{\dom}$ such that for every $k \in \N$, $1\leq k \leq n$ 
	\begin{equation*}
		\cor{\norme_k} \subset \bd{\dom}
	\end{equation*}
	and for all $\funv \in \dmid{\dom}$ the following \textit{Gauß
	formula}\index{Gauß formula} holds
	\begin{equation*}
		\sI{\bd{\dom}}{\funv}{\norme} = \divv{\funv}(\dom) \,.
	\end{equation*}
	and for every open set $\bals \subset \Rn$
	\begin{equation*}
		\tova{\norme}(\bals \cap \dom) \geq
		(\reme{\ham^{n-1}}{\rbd{\dom}})(\bals) \,.
	\end{equation*}
	Furthermore, $\norme$ is minimal in the sense, that if $\norme' \in
	\bawln{\dom}$ satisfies the equations above, then
	\begin{equation*}
		\norm{}{\norme} \leq \norm{}{\norme'} \,.
	\end{equation*}
	For every $\bvs\in \bor{\dom}$ having finite perimeter in $\Rn$
	\begin{equation*}
		\norme(\bvs) = -\I{\rbd{\bvs}\cap
		\dom}{\normal{\bvs}}{\ham^{n-1}} \,.
	\end{equation*}
\end{theorem}
\begin{remark}
	Note that if $\dom\in\bor{\dom}$ is only supposed to be open and 
	\begin{equation*}
		\ham^{n-1}(\bd{\idnhd{\dom}{\delta}}) \leq \const
	\end{equation*}
	is required, then $\dom$ is necessarily a \tbvs{}, due to the \ttova{}
	being lower semi-continuous.

	On the other hand, this condition loosely resembles the definition of
	Lipschitz deformable boundaries defined in
	\cite[p. 94]{chen_divergencemeasure_1999}, but is much more general.
\end{remark}
\begin{proof}
	Let $\{\nap_k\}_{k\in \N} \subset \skpdr{1}{\infty}{\Rn}{[0,1]}$
	be such that
	\begin{align*}
		\nap_k & := \ind{\idnhd{\dom}{\frac{2}{k}}} +
		\ind{\left(\idnhd{\dom}{\frac{1}{k}}\setminus\idnhd{\dom}{\frac{2}{k}}\right)}
		\left(k\distf{\bd{\dom}}-1 \right)  \\
		& = \min\left\{1,\max\left\{0,k\distf{\bd{\dom}}-1\right\}\right\}\,.
	\end{align*}
	See \cite[p. 47]{clarke_optimization_1984} for reference.
	Then
	\begin{equation*}
		|\Deriv{\nap_k}| =
		k
		\ind{\left(\idnhd{\dom}{\frac{1}{k}}\setminus\idnhd{\dom}{\frac{2}{k}}\right)}
		\,.
	\end{equation*}
	Then the Coarea Formula \cite[p. 112]{Evans1992} implies
	\begin{equation*}
		\norm{1}{\Deriv{\nap_k}} =
		\I{\idnhd{\dom}{\frac{1}{k}}\setminus\idnhd{\dom}{\frac{2}{k}}}{k}{\lem}
		=
		\mI{(\frac{1}{k},\frac{2}{k})}{\ham^{n-1}(\bd{\idnhd{\dom}{\delta}})}{\delta}\leq
		\const \,.
	\end{equation*}
	As in the proof of Theorem \ref{thm:ggt_me}, 
	\begin{equation*}
		\lim \limits_{k \to \infty} \I{\dom}{\funv \cdot
		\Deriv{\nap_k}}{\lem} = - \lim \limits_{k \to \infty}
		\I{\dom}{\nap_k}{\divv{\funv}} = -\I{\dom}{1}{\divv{\funv}} =
		\divv{\funv}(\dom) \,.
	\end{equation*}
	On the other hand, for every $k\in \N$
	\begin{equation*}
		\left | \I{\dom}{\funv\cdot \Deriv{\nap_k}}{\lem}\right | \leq
		\normi{\funv} \norm{1}{\Deriv{\nap_k}} \leq \normi{\funv} \sup \limits_{k \in
		\N} \norm{1}{\Deriv{\nap_k}} \leq \const \normi{\funv} \,.
	\end{equation*}
	Hence
	\begin{equation*}
		\de_0 : \dmid{\dom} \to \R: \funv \mapsto \divv{\funv}(\dom)
	\end{equation*}
	is a continuous linear functional on a subspace of $\liblnd{\dom}$.
	\begin{change}
		The Hahn-Banach Theorem (cf. \cite[p.63]{dunford_linear_1988})
		implies the existence of a \tme{} $\norme \in \bawln{\dom}$
		such that for all $\funv \in \dmi$
		\begin{equation}\label{eq:ggt_me_all}
			\divv{\funv}(\dom) = \I{\dom}{\funv}{\norme} \,.
		\end{equation}
		Furthermore, $\norm{}{\norme} = \norm{}{\de_0}$, implying
		minimality in the norm. Now by
		\cite[p. 101]{chen_divergencemeasure_1999}, for almost every $\delta > 0$ and every
		$\funv \in \dmid{\dom}$
		\begin{equation*}
			\funv \cdot \ind{\idnhd{\dom}{\delta}} \in \dmid{\dom}
		\end{equation*}
		and $\funv \cdot \ind{\idnhd{\dom}{\delta}}$ has compact
		support in $\dom$. By \cite[p. 252]{chen_divergence-measure_2005}
		\begin{equation*}
			\divv{(\funv \cdot \ind{\idnhd{\dom}{\delta}})}(\dom) =
			0 \,.
		\end{equation*}
		Thus, for every $\funv \in \dmid{\dom}$
		\begin{align*}
			\I{\dom}{\funv}{\norme} & = \divv{\funv}(\dom) \\
						& = \divv{(\funv \cdot
		\ind{\dom\setminus\idnhd{\dom}{\delta}})}(\dom)
		+\divv{(\funv\cdot \ind{\idnhd{\dom}{\delta}})}(\dom) \\
		& = \divv{(\funv\cdot \ind{\dom\setminus
\idnhd{\dom}{\delta}})}(\dom) \\
& = \I{\dom}{\funv}{\reme{\norme}{(\dom\setminus\idnhd{\dom}{\delta})}} \,.
		\end{align*}
		Thus, $\reme{\norme}{(\dom\setminus \idnhd{\dom}{\delta})}$
		also satisfies Equation \eqref{eq:ggt_me_all}. The minimality
		of $\norm{}{\norme}$ then implies
		\begin{equation*}
			\norm{}{\reme{\norme}{\idnhd{\dom}{\delta}}} = 0 \,.
		\end{equation*}
		Since $\delta > 0$ can be arbitrarily small
		\begin{equation*}
			\cor{\norme} \subset \bd{\dom} \,.
		\end{equation*}
		Note that for $\bals \in \bor{\dom}$ having finite perimeter in $\Rn$
		\begin{equation*}
			e_k \ind{\bals} \in \dmid{\dom} \,.
		\end{equation*}
		In order to see this, compute
		\begin{equation*}
			\divv{(e_k \cdot \ind{\bvs})} = \pD{\ind{\bvs}}{k} =
			-\reme{\acme{\normal{\bvs}_k}{\ham^{n-1}}}{\rbd{\bvs}} \,.
		\end{equation*}
		In particular
		\begin{equation*}
			\norme(\bvs) = -\I{\rbd{\bvs}\cap
			\dom}{\normal{\bvs}}{\ham^{n-1}} \,.
		\end{equation*}
		Now, let $\bals \subset\Rn$ be open. Then using the Gauß
		Theorem from Evans (cf. \cite[p.209]{Evans1992})
		\begin{align*}
			\tova{\norme}(\bals \cap \dom) & \geq \sup
			\limits_{\substack{\sfun \in
				\ConecN{\bals},\\\normi{\sfun}\leq 1}}
				\I{\dom}{\sfun}{\norme} \\
				& \geq \sup \limits_{\substack{\sfun \in
			\ConecN{\bals},\\\norm{C}{\sfun}\leq 1}}
			\I{\dom}{\sfun}{\norme} \\
			& \overset{\text{\eqref{eq:ggt_me_all}}}{=}\sup
			\limits_{\substack{\sfun \in \ConecN{\bals},\\\norm{C}{\sfun}\leq 1}} \divv{\sfun}(\dom) \\
				& = \sup \limits_{\substack{\sfun \in
			\ConecN{\bals},\\\norm{C}{\sfun}\leq 1}}
			\I{\rbd{\dom}}{\sfun\cdot \normal{\dom}}{\ham^{n-1}} \\
				& = \tova{\Deriv{\ind{\dom}}}(\bals) \\
				& = (\reme{\ham^{n-1}}{\rbd{\dom}})(\bals) \,.
		\end{align*}
	\end{change}
\end{proof}
\begin{remark}
	Theorem \ref{thm:ggt_me_all} still holds true for open $\dom$ such that
	there exists $\{\delta_k\}_{k \in \N}\subset (0,\infty)$ such that
	$\lim \limits_{k \to \infty} \delta_k = 0$ and
	\begin{equation*}
		\sup \limits_{k \in \N}
		\mI{\left(\frac{\delta_k}{2},\delta_k\right)}{\ham^{n-1}(\idnhd{\dom}{\delta})}{\delta} <
		\infty \,.
	\end{equation*}
	The arguments are the same as in Example \ref{ex:outer_port} and
	Example \ref{ex:inner_port}.
\end{remark}
The following proposition is a new Gauß-Green formula for essentially bounded
functions of bounded variation and essentially bounded vector fields having
divergence measure. In contrast to the literature, where only continuous scalar
fields were treated (cf. \cite[p. 448]{silhavy_divergence_2009},
\cite[p. 1014]{chen_structure_2011}), this is a new quality.
\begin{proposition}\label{prop:ntfunv_dmi}
	Let $\encl{\dom}\subset \Rn$ be open and $\dom \in \bor{\encl{\dom}}$
	be a bounded \tbvs{} such that $\dist{\dom}{\bd{\encl{\dom}}}>0$.
	Furthermore, let $\{\nap_k\}_{k \in \N} \subset
	\skpdr{1}{\infty}{\encl{\dom}}{[0,1]}$ be a \tnormapprox{} for
	$\ind{\dom}$
	and $\norme \in \bawln{\encl{\dom}}$ be an associated \tggtme{}.
	
	Then for every $\funv \in \dmi$ the set function
	\begin{equation*}
		\ntfunv{\funv} : \bor{\encl{\dom}} \to \R: \bals \mapsto
		\I{\bals}{\funv}{\norme}
	\end{equation*}
	is an element of $\bawl{\encl{\dom}}$ with 
	\begin{equation*}
		\cor{\ntfunv{\funv}} \subset \bd{\dom}
	\end{equation*}
	and for every compactly supported $\bvfun \in
	\BVd{\encl{\dom}} \cap \libld{\encl{\dom}}$ the following \textit{Gauß
	formula}\index{Gauß formula} holds
	\begin{equation*}
		\divv{(\bvfun \cdot \funv)}(\Int{\dom}) +
		\I{\bd{\dom}}{\nap}{\divv{(\bvfun \cdot \funv)}} =
		\sI{\bd{\dom}}{\bvfun}{\ntfunv{\funv}} =
		\sI{\bd{\dom}}{\bvfun \cdot \funv}{\norme} \,.
	\end{equation*}
	Call $\ntfunv{\funv}$ \textbf{normal trace of $\funv$ on $\bd{\dom}$}.
\end{proposition}
\begin{proof}
	Note that	
	\begin{equation*}
		\bvfun \cdot \funv \in \dmi \,.
	\end{equation*}
	See \cite[p. 1014]{chen_structure_2011} for reference.
	Hence
	\begin{equation*}
		\divv{(\bvfun \cdot \funv)}(\Int{\dom}) +
		\I{\bd{\dom}}{\nap}{\divv{(\bvfun \cdot \funv)}} =
		\sI{\bd{\dom}}{\bvfun \cdot \funv}{\norme}\,.
	\end{equation*}
	Note that for every $\bals \in \bor{\encl{\dom}}$
	\begin{equation*}
		\left |\I{\bals}{\funv}{\norme}\right | \leq \normi{\funv}
		\tova{\norme}(\bals)  \,,
	\end{equation*}
	whence
	\begin{equation*}
		\ntfunv{\funv} \in \bawl{\encl{\dom}}\,.	
	\end{equation*}
	Since for every $\bals \in \bor{\encl{\dom}}$
	\begin{equation*}
		\ntfunv{\funv}(\bals) = \I{\bals}{\funv}{\norme} = \I{\bals
			\cap (\dnhd{\dom}{\delta} \setminus
		\idnhd{\dom}{\delta})}{\funv}{\norme}
	\end{equation*}
	the \tcor{} of $\ntfunv{\funv}$ is a subset of $\bd{\dom}$.

	Let $\eps >0$. Since $\fun \in \libld{\encl{\dom}}$, there exist $m
	\in \N$, $\{\eRR_k\}_{k=\mN}^m \subset \R$ and $\{\bals_k\}_{k=\mN}^m$
	pairwise disjoint, such that
	\begin{equation*}
		\normi{\eRR_k - \fun\cdot \ind{\bals_k}}\leq \eps \text{ and } \bigcup\limits_{k=\mN}^m \bals_k = \encl{\dom}
	\end{equation*}
	Set $\simf := \sum \limits_{k=\mN}^m \eRR_k \ind{\bals_k}$. Then
	\begin{align*}
		\left |\sI{\bd{\dom}}{\fun \funv}{\norme} -
		\sI{\bd{\dom}}{\fun}{\ntfunv{\funv}} \right | & \leq \left |
		\sI{\bd{\dom}}{(\fun - \simf)\funv}{\norme} \right | + \left
		|\sI{\bd{\dom}}{\simf \funv}{\norme} -
		\sI{\bd{\dom}}{\simf}{\ntfunv{\funv}} \right | \\
			& + \left | \sI{\bd{\dom}}{\fun -
		\simf}{\ntfunv{\funv}} \right | \\
		& \leq \eps \normi{\funv}\tova{\norme}(\encl{\dom}) + 0 + \eps
		\tova{\ntfunv{\funv}}(\encl{\dom}) \,.
	\end{align*}
	Since $\eps >0$ was arbitrary
	\begin{equation*}
		\sI{\bd{\dom}}{\fun\funv}{\norme} =
		\sI{\bd{\dom}}{\fun}{\ntfunv{\funv}} \,.	
	\end{equation*}
\end{proof}
\section{Unbounded Vector Fields and Open Sets}
In the previous section, general Gauß formulas for essentially bounded vector
fields having divergence measure were presented. Example \ref{ex:tang_unbound}
and \ref{ex:dirac_mass} showed that it is in general not possible to integrate
unbounded vector fields with respect to the normal measures obtained.
In Proposition \ref{prop:ntfunv_dmi}, the \tme{} $\ntfunv{\funv}$ was presented
as a notion of normal trace.

In the following, this is carried over to the case of unbounded vector
fields. Therefore, a result due to Silhavy (cf. \cite{silhavy_divergence_2009})
is improved upon. In particular, Silhavy proved that for $\funv \in
\dmpd{1}{\encl{\dom}}$ there exists a continuous linear functional on
$\Lcspace{\bd{\dom}}$, the space of \tLcont{} functions on $\bd{\dom}$, balancing
the volume part of the Gauß formula. The following exposition proves that
this functional can be represented by the sum of a \tRaM{} $\ntfunv{\funv}$ and a \tme{}
$\ntfunvd{\funv} \in \bawln{\encl{\dom}}$ with \tcor{} on the boundary.
The arguments from Silhavy are retraced, in order to give a self-contained
proof of the main theorem.

Throughout this section, for $\funv \in (\libld{\encl{\dom}})^n$ and
$\nhd\subset \encl{\dom}$ open, set
\begin{equation*}
	\norm{\infty,\nhd}{\funv} := \essup{\nhd} |\funv| \,.	
	\nomenclature[n]{$\norm{\infty,\nhd}{\funv}$}{$\essup{\nhd}{\tova{\funv}}$}
\end{equation*}
It is essential for the subsequent proofs to be able to compare the Lipschitz
constant of a function by the norm of its gradient. The following lemma enables
this comparison on balls.
\begin{lemma}\label{lem:bd_lip_ball}
	Let $\encl{\dom}\subset \Rn$ be open and $\fun \in
	\skpdr{1}{\infty}{\encl{\dom}}{\R}$. Then for every $\eR_0 \in
	\encl{\dom}$ and $0 < \delta < \frac{1}{2}\dist{\eR_0}{\bd{\encl{\dom}}}$ with $\ball{\eR_0}{\delta}\subset
	\encl{\dom}$
	\begin{equation*}
		\sup\limits_{\substack{\eR,\eRR\in\ball{\eR_0}{\delta}\\\eR \neq
		\eRR}} \frac{|\fun(\eR) - \fun(\eRR)|}{|\eR- \eRR|} \leq
	\norm{\infty,\ball{\eR_0}{2\delta}}{\Deriv{\fun}} \,.
	\end{equation*}
\end{lemma}
\begin{proof}
	Let $\eps < \delta$. For $\eR \in
	\ball{\eR_0}{\delta}$ set
	\begin{equation*}
		\fun_\eps(\eR) := \I{\Rn}{\mol_\eps(\eRR-\eR) \fun(\eRR)}{\eRR} =
		\mol_\eps * \fun (\eR)
		\,,
	\end{equation*}
	where $\mol_\eps$ is a scaled standard mollification kernel. Then as in
	Evans \cite[p. 123]{Evans1992}
	\begin{equation*}
		\Deriv{\fun_\eps} = \mol_\eps * \Deriv{\fun}\,.
	\end{equation*}
	Note that $\fun_\eps \to \fun$ point wise (cf.
	\cite[p. 123]{Evans1992}).
	Hence, for every $\eR,\eRR\in \ball{\eR_0}{\delta}$ with $\eR \neq
	\eRR$
	\begin{align*}
		|\fun(\eR) - \fun(\eRR)| & = \lim \limits_{\eps \downarrow 0}
		|\fun_\eps(\eR) - \fun_\eps(\eRR)| \\
		& \leq \liminf\limits_{\eps \downarrow 0}
		\normi{\Deriv{\fun_\eps}} |\eR- \eRR| \,.
	\end{align*}
	Now for every $\eR \in \ball{\eR_0}{\delta}$
	\begin{equation*}
		|\Deriv{\fun_\eps}(\eR)| \leq \I{\Rn}{|\mol_\eps(\eRR-\eR)|
			|\Deriv{\fun}(\eRR)|}{\lem} \leq
			\norm{\infty,\ball{\eR_0}{\delta + \eps}}{\Deriv{\fun}} \,.
	\end{equation*}
	Thus
	\begin{equation*}
		\sup\limits_{\substack{\eR,\eRR \in
			\ball{\eR_0}{\delta}\\\eR\neq \eRR}} \frac{|\fun(\eR) -
			\fun(\eRR)|}{|\eR-\eRR|} \leq \liminf\limits_{\eps
			\downarrow 0}\norm{\infty,\ball{\eR_0}{\delta +
				\eps}}{\Deriv{\fun}} \leq
				\norm{\infty,\ball{\eR_0}{2\delta}}{\Deriv{\fun}}
	\end{equation*}
\end{proof}
Once the estimate on balls is obtained, it is possible to prove the statement
for path-connected sets.
\begin{lemma}\label{lem:bd_lip_path}
	Let $\encl{\dom}\subset \Rn$ be open and $\ferm \subset \encl{\dom}$ be
	compact and path-connected. Furthermore let $0< \delta <\dist{\ferm}{\bd{\encl{\dom}}}$.

	Then there exists $\const >0$ depending only on $\ferm$ and $\delta$ such
	that for every $\fun \in \skpdr{1}{\infty}{\encl{\dom}}{\R}$
	\begin{equation*}
		\sup \limits_{\substack{\eR,\eRR\in
			\ferm\\\eR\neq \eRR}}
			\frac{|\fun(\eR)-\fun(\eRR)|}{|\eR-\eRR|} \leq \const
			\norm{\infty,\dnhd{\ferm}{\delta}}{\Deriv{\fun}} \,.
	\end{equation*}
\end{lemma}
\begin{proof}
	First, let $\delta < \frac{1}{6} \dist{\ferm}{\bd{\encl{\dom}}}$.
	Note that
	\begin{equation*}
		\ferm \subset  \bigcup\limits_{\eR \in \ferm}
		\ball{\eR}{\delta} \,.
	\end{equation*}
	Since $\ferm$ is compact, there exists $m \in \N$ and
	$\{\eR_k\}_{k=\mN}^m\subset \ferm$ such that
	\begin{equation*}
		\ferm \subset \bigcup \limits_{k=\mN}^m \ball{\eR_k}{\delta}
		\,.
	\end{equation*}
	Now, let $\eR,\eRR \in \ferm$ with $\eR\neq \eRR$ be such that $|\eR-\eRR| <\delta$. 
	Then
	\begin{equation*}
		\eRR \in \ball{\eR}{\delta} \subset \ball{\eR}{2\delta}\subset
		\dnhd{\ferm}{2\delta} \subset \encl{\dom} \,.
	\end{equation*}
	By Lemma \ref{lem:bd_lip_ball}
	\begin{equation*}
		\frac{|\fun(\eR) - \fun(\eRR)|}{|\eR-\eRR|} \leq
		\norm{\infty,\ball{\eR}{2\delta}}{\Deriv{\fun}} \leq
		\norm{\infty,\dnhd{\ferm}{2\delta}}{\Deriv{\fun}} \,.
	\end{equation*}
	Now, assume $|\eR-\eRR|\geq \delta$. Then there exists a continuous $\path : [0,1]
	\to \ferm$ such that $\path(0) = \eR$ and $\path(1) = \eRR$. Let $\mN\leq
	k \leq m$ such that
	\begin{equation*}
		\eR \in \ball{\eR_k}{\delta}	
	\end{equation*}
	and set
	\begin{equation*}
		t_\mN := \sup\{t\in [0,1] \setpipe \path(t) \in \ball{\eR_k}{\delta}\}
		\,.
	\end{equation*}
	If $\path(t_\mN) \in \ball{\eR_k}{\delta}$, then $t_\mN =1$ and
	$\eRR\in \ball{\eR_k}{\delta}$. Hence
	\begin{equation*}
		\frac{|\fun(\eR)- \fun(\eR)|}{|\eR-\eRR|} \leq
		\norm{\infty,\ball{\eR_k}{2\delta}}{\Deriv{\fun}} \leq
		\norm{\infty,\dnhd{\ferm}{2\delta}}{\Deriv{\fun}} \,.
	\end{equation*}
	Otherwise, $\path(t_\mN) \notin\ball{\eR_k}{\delta}$. But then there
	exists $\mN\leq l \leq m$ and $l \neq k$ such that
	\begin{equation*}
		\path(t_\mN) \in \ball{\eR_l}{\delta}
	\end{equation*}
	and 
	\begin{equation*}
		\path(t) \notin \ball{\eR_k}{\delta} \text{ for all } t \geq
		t_\mN \,.
	\end{equation*}
	Set 
	\begin{equation*}
		\path^\mN(t) = \begin{cases}
			\eR + \frac{t}{t_0}(\path(t_\mN) - \eR) &\mbox{ for } t \leq t_\mN
			\\
			\path(t) &\mbox{ otherwise.}
		\end{cases}	
	\end{equation*}
	In essence, $\path^{\mN}$ is a shortcut in $\cl{\ball{\eR_k}{\delta}}$ to the last point where $\path$
	is in $\cl{\ball{\eR_k}{\delta}}$.

	Repeating the steps above, induction yields a continuous path
	$\overline{\path} :
	[0,1] \to \cl{\dnhd{\ferm}{\delta}}$, $\mN \leq m' \leq m$ and $\{t_l\}_{l=1}^{m'}\subset
	[0,1]$ such that
	\begin{equation*}
		|\overline{\path}(t_l) - \overline{\path}(t_{l+1})| \leq 2\delta \text{ for } l=\mN,...,m'-1
	\end{equation*}
	and 
	\begin{equation*}
		|\eR-\overline{\path}(t_\mN)|\leq 2\delta \text{ and }
		|\eRR-\overline{\path}(t_{m'})|\leq 2\delta
		\,.
	\end{equation*}
	Using Lemma \ref{lem:bd_lip_ball} again for balls of radius $3\delta$
	\begin{align*}
		|\fun(\eR) - \fun(\eRR)| \leq & |\fun(\eR) -
		\fun(\overline{\path}(t_\mN)|
		+ \sum \limits_{l=\mN}^{m'-1} | \fun(\overline{\path}(t_l)) -
		\fun(\overline{\path}(t_{l+1}))| + |\fun(\overline{\path}_{m'}) - \fun(\eRR)|\\
		\leq & \norm{\infty,\ball{\eR}{6\delta}}{\Deriv{\fun}}
		|\eR-\overline{\path}(t_\mN)| + \sum \limits_{l=\mN}^{m'-1}
		\norm{\infty,\ball{\overline{\path}(t_l)}{6\delta}}{\Deriv{\fun}}
		|\overline{\path}(t_l)-\overline{\path}(t_{l+1})| \\
		& + \norm{\infty,\ball{\eRR}{6\delta}}{\Deriv{\fun}}
		|\eRR-\overline{\path}(t_{m'})| \\
		\leq & \norm{\infty,\dnhd{\ferm}{6\delta}}{\Deriv{\fun}}
		2\delta (m'+2) \\
		\leq &
		2(m+2)\norm{\infty,\dnhd{\ferm}{6\delta}}{\Deriv{\fun}}|\eR-\eRR|
		\,.
	\end{align*}
	Since $\eR,\eRR\in \ferm$ were arbitrary
	\begin{equation*}
		\sup \limits_{\substack{\eR,\eRR\in \ferm\\\eR\neq \eR}}
		\frac{|\fun(\eR) - \fun(\eRR)|}{|\eR-\eRR|} \leq 2(m+2)
		\norm{\infty,\dnhd{\ferm}{6\delta}}{\Deriv{\fun}} \,.
	\end{equation*}
	Note that $m$ only depends on $\ferm$ and $\delta$.
	Finally, for $0< \delta < \dist{\ferm}{\bd{\encl{\dom}}}$ set
	\begin{equation*}
		\overline{\delta} := \frac{1}{6}{\delta} \,.
	\end{equation*}
	Then the inequality above yields
	\begin{equation*}
		\sup \limits_{\substack{\eR,\eRR\in \ferm\\\eR\neq \eRR}}
		\frac{|\fun(\eR)-\fun(\eRR)|}{|\eR-\eRR|} \leq 2(m+2)
		\norm{\infty,\dnhd{\ferm}{\delta}}{\Deriv{\fun}} \,.
	\end{equation*}
	This finishes the proof.
\end{proof}
\begin{remark}
	The requirement that $\ferm$ is path-connected cannot be dropped. In
	order to see this, let $\encl{\dom} := \R^2$ and 
	\begin{equation*}
		\ferm := [-1,1]\times \{-2,2\} \,.
	\end{equation*}
	Let $\fun\in \skpdr{1}{\infty}{\R^2}{\R}$ be a \tLcont{} function such that
	\begin{equation*}
		\fun := \begin{cases}
		1 &\mbox{ on } [-2,2]\times [1,3] \\
		0 &\mbox{ on } [-2,2]\times [-3,-1] \,.
		\end{cases}
	\end{equation*}
	Then for $0 < \delta < 1$
	\begin{equation*}
		\norm{\infty,\dnhd{\ferm}{\delta}}{\Deriv{\fun}} = 0 
	\end{equation*}
	but
	\begin{equation*}
		\sup \limits_{\substack{\eR,\eRR\in \ferm\\\eR\neq \eRR}}
		\frac{|\fun(\eR) - \fun(\eRR)|}{|\eR-\eRR|} > 0 \,.
	\end{equation*}
\end{remark}
Since the trace operator of Silhavy \cite{silhavy_divergence_2009} is defined
on the space of \tLcont{} functions, this space needs to be introduced now.
\begin{definition}
	Let $\dom \subset \Rn$. Let 
	\begin{equation*}
		\Lcspace{\dom}
		\nomenclature[f]{$\Lcspace{\dom}$}{\tLcont{} functions on $\dom$}
	\end{equation*}
	denote the set of all \tLcont{} functions on $\dom$. For $\fun \in
	\Lcspace{\dom}$ set
	\begin{equation*}
		\Lcnorm{\fun} := \norm{C}{\fun} + \sup
		\limits_{\substack{\eR,\eRR\in \dom,\\\eR\neq \eRR}}
		\frac{|\fun(\eR) - \fun(\eRR)|}{|\eR - \eRR|} \,.
		\nomenclature[n]{$\Lcnorm{\fun}$}{$\norm{C}{\fun} + \sup
		\limits_{\substack{\eR,\eRR\in \dom,\\\eR\neq \eRR}}
		\frac{\tova{\fun(\eR) - \fun(\eRR)}}{\tova{\eR - \eRR}}$}
	\end{equation*}
\end{definition}
The following result is a slight variation of Lemma 3.2 in
Silhavy \cite[p. 451]{silhavy_divergence_2009}. It states that the Gauß formula
yields zero, if the scalar field is zero on the boundary.
\begin{proposition}\label{prop:lip_bd_zero}
	Let $\dom \subset \Rn$ be open and bounded, $\funv \in \dmpd{1}{\dom}$ and $\fun \in \Lcspace{\cl{\dom}}$ be
	such that
	\begin{equation*}
		\refun{\fun}{\bd{\dom}} = 0 \,.
	\end{equation*}
	Then
	\begin{equation*}
		\I{\dom}{\funv \cdot \Deriv{\fun}}{\lem} +
		\I{\dom}{\fun}{\divv{\funv}} = 0 \,.
	\end{equation*}
\end{proposition}
\begin{proof}
	First, suppose that 
	\begin{equation*}
		\supp{\fun} \csubset \dom\,.	
	\end{equation*}
	By \cite[p. 252]{chen_divergence-measure_2005} and Silhavy
	 \cite[p. 448]{silhavy_divergence_2009} (cf. \cite[p.
	 250]{chen_divergence-measure_2005})
	\begin{equation*}
		\I{\dom}{1}{\divv{(\funv\cdot\fun)}}
		=\I{\dom}{\fun}{\divv{\funv}} + \I{\dom}{\funv\cdot
		\Deriv{\fun}}{\lem} = 0 \,.
	\end{equation*}
	For the general case, let 
	\begin{align*}
		\nap_k & := \ind{\idnhd{\dom}{\frac{2}{k}}} + \left(k \distf{\bd{\dom}}-1\right) \ind{\idnhd{\dom}{\frac{1}{k}}\setminus
			\idnhd{\dom}{\frac{2}{k}}}\\
			& = \min\left\{1,\max\left\{0,k\distf{\bd{\dom}} -
	1\right\}\right\} \in \Lcspace{\cl{\dom}} \,.
	\end{align*}
	Then $\fun \cdot \nap_k \in \Lcspace{\cl{\dom}}$ (cf.
	\cite[p. 48]{clarke_optimization_1984}). In order to estimate the norm
	independently of $k\in \N$,
	let $\eR,\eRR \in \dom$. If $\eR,\eRR \in \idnhd{\dom}{\frac{2}{k}}$
	then
	\begin{equation*}
		|\fun(\eR)\nap_k(\eR) - \fun(\eRR)\nap_k(\eRR)| = |\fun(\eR) -
		\fun(\eRR)| \leq \Lcnorm{\fun}|\eR-\eRR| \,.
	\end{equation*}
	Otherwise, w.l.o.g. $\eR \in \dom\setminus\idnhd{\dom}{\frac{2}{k}}$
	and
	\begin{align*}
		|\fun(\eR)\cdot \nap_k(\eR) - \fun(\eRR)\nap_k(\eRR)| & \leq
		|\fun(\eR)| |\nap_k(\eR) -\nap_k(\eRR)| +
		|\nap_k(\eRR)||\fun(\eR) - \fun(\eRR)| \\
		& \leq \norm{C\left(\dom\setminus\idnhd{\dom}{\frac{2}{k}}\right)}{\fun} |\nap_k(\eR) -\nap_k(\eRR)| + |\fun(\eR)
		- \fun(\eRR)|\\
		& \leq
		\left(\sup\limits_{0\leq \distf{\bd{\dom}}(\eR)
		\leq \frac{2}{k}} |\fun(\eR)| k + \Lcnorm{\fun}\right) |\eR - \eRR|
		\,.
	\end{align*}
	Since $\fun$ vanishes on $\bd{\dom}$
	\begin{equation*}
		\sup\limits_{0\leq \distf{\bd{\dom}}(\eR)\leq
		\frac{2}{k}} |f(x)- 0| \leq \Lcnorm{\fun} \frac{2}{k} \,,
	\end{equation*}
	whence
	\begin{equation*}
		\Lcnorm{\fun \cdot \nap_k} \leq 3\Lcnorm{\fun} \,.
	\end{equation*}
	Furthermore, for every $k \in \N$ 
	\begin{equation*}
		\supp{\fun\cdot \nap_k} \csubset \dom \,.
	\end{equation*}
	Hence, for every $k \in \N$
	\begin{equation*}
		\I{\dom}{\funv \cdot \Deriv{(\fun \cdot \nap_k)}}{\lem} +
		\I{\dom}{\fun \cdot \nap_k}{\divv{\funv}} = 0 \,.
	\end{equation*}
	First note that
	\begin{equation*}
		\I{\dom}{\fun\cdot \nap_k}{\divv{\fun}} \xrightarrow{k \to
		\infty} \I{\dom}{\fun}{\divv{\funv}}
	\end{equation*}
	by the Dominated Convergence Theorem (cf. \cite[p. 20]{Evans1992}).
	
	On the other hand, since $\normi{\Deriv{(\fun\cdot \nap_k)}}\leq
	\Lcnorm{\fun\cdot \nap_k}$ is bounded
	independently of $k \in \N$ the Dominated Convergence Theorem also
	yields
	\begin{equation*}
		\I{\dom}{\funv\cdot \Deriv{(\fun\cdot \nap_k)}}{\lem}
		\xrightarrow{k \to \infty} \I{\dom}{\funv \cdot
		\Deriv{\fun}}{\lem}	\,.
	\end{equation*}
	Hence
	\begin{equation*}
		\I{\dom}{\funv\cdot \Deriv{\fun}}{\lem} +
		\I{\dom}{\fun}{\divv{\funv}} \xleftarrow{k \to\infty}
		\I{\dom}{\funv\cdot \Deriv{(\fun \cdot \nap_k)}}{\lem} +
		\I{\dom}{\fun \cdot \nap_k}{\divv{\funv}} = 0 \,.
	\end{equation*}
\end{proof}
The following proposition is a specialised version of Theorem 2.3 in Silhavy
\cite[p. 448]{silhavy_divergence_2009}. It states that the volume part of a
Gauß formula only depends on the boundary values of the \tLcont{} scalar
function.
\begin{proposition}\label{prop:ntfs_sil}
	Let $\dom \subset \Rn$ be open and bounded and $\funv \in
	\dmpd{1}{\dom}$. Then there exists a continuous linear functional
	\begin{equation*}
		\ntfs{\funv}{\dom}:\Lcspace{\bd{\dom}} \to \R
	\end{equation*}
	such that for every $\fun \in \Lcspace{\cl{\dom}}$
	\begin{equation*}
		\ntfs{\funv}{\dom}(\refun{\fun}{\bd{\dom}}) =
		\I{\dom}{\fun}{\divv{\funv}} +
		\I{\dom}{\funv\cdot\Deriv{\fun}}{\lem} \,.
	\end{equation*}
	Furthermore
	\begin{equation*}
		\norm{}{\ntfs{\funv}{\dom}} \leq \normdmp{1}{\funv} \,.	
	\end{equation*}
\end{proposition}
\begin{proof}
	The proof follows the same lines as the one in \cite[p. 452]{silhavy_divergence_2009}.
	Let $\fun \in \Lcspace{\bd{\dom}}$ and $\fun_1,\fun_2 \in
	\Lcspace{\Rn}$ be extensions of $\fun$ to all of $\Rn$ (cf.
	\cite[p. 201]{federer_geometric_1996}). 
	Note that $\refun{(\fun_1-\fun_2)}{\bd{\dom}}= 0$.
	Then by Proposition \ref{prop:lip_bd_zero} 
	\begin{equation*}
		\I{\dom}{\fun_1}{\divv{\funv}} + \I{\dom}{\funv \cdot
			\Deriv{\fun_1}}{\lem} = \I{\dom}{\fun_2}{\divv{\funv}}
			+ \I{\dom}{\funv \cdot \Deriv{\fun_2}}{\lem} \,.
	\end{equation*}
	For $\fun\in \Lcspace{\bd{\dom}}$ and any extension $\funex{\fun} \in
	\Lcspace{\Rn}$ of $\fun$ define
	\begin{equation*}
		\ntfs{\funv}{\dom} (\fun)
		:=\I{\dom}{\funex{\fun}}{\divv{\funv}} +
		\I{\dom}{\funv\cdot\Deriv{\funex{\fun}}}{\lem} \,.
	\end{equation*}
	Then $\ntfs{\funv}{\dom} : \Lcspace{\bd{\dom}} \to \R$ is well-defined
	and a linear functional. For $\fun \in \Lcspace{\bd{\dom}}$ there exists an extension
	$\funex{\fun}\in \Lcspace{\Rn}$ such that
	\begin{equation*}
		\Lcnorm{\funex{\fun}} = \Lcnorm{\fun} \,.
	\end{equation*}
	See Silhavy \cite[p. 452]{silhavy_divergence_2009} and Federer \cite[p. 201]{federer_geometric_1996} for reference.
	With this extension
	\begin{align*}
		|\ntfs{\funv}{\dom}(\fun)| & \leq \tova{\divv{\funv}}(\dom)
		\norm{C}{\funex{\fun}} + \norm{1}{\funv}
		\normi{\Deriv{\funex{\fun}}} \\
		& \leq \normdmp{1}{\funv}
		\Lcnorm{\funex{\fun}} \\
		& = \normdmp{1}{\funv} \Lcnorm{\fun} \,.
	\end{align*}
\end{proof}
Up to now, the arguments from Silhavy \cite{silhavy_divergence_2009} were
retraced. Now, the representation of $\ntfs{\funv}{\dom}$ by the sum of a \tRaM{}
and a \tme{} $\ntfunvd{\funv}\in \bawln{\encl{\dom}}$ is proved. This
result is new because it gives the abstract functionals found in the literature
a concrete representation as integral functionals.
\begin{theorem}{Gauß Theorem\\}\label{thm:ggt_ubound}
Let $\encl{\dom}\subset \Rn$ be open, $\dom \subset \encl{\dom}$ be open with
$\cl{\dom}\subset \encl{\dom}$ compact and $\bd{\dom}$ path-connected. Furthermore, let $\funv \in
\dmpd{1}{\encl{\dom}}$.

Then there exists a \tRaM{}
$\ntfunv{\funv}$\nomenclature[f]{$\ntfunv{\funv}$}{\tRaM{} from Gauß Theorem} on $\bd{\dom}$ and
$\ntfunvd{\funv} \in
\bawln{\encl{\dom}}$\nomenclature[f]{$\ntfunvd{\funv}$}{measure from Gauß
Theorem} with 
\begin{equation*}
	\cor{\ntfunvd{\funv}}\subset \bd{\dom}	
\end{equation*}
such that for all $\fun \in \skpdr{1}{\infty}{\encl{\dom}}{\R}$ the following
\textit{Gauß-Green formula}\index{Gauß formula} holds
\begin{equation*}
	\I{\bd{\dom}}{\fun}{\ntfunv{\funv}} +
	\sI{\bd{\dom}}{\Deriv{\fun}}{\ntfunvd{\funv}} =
	\I{\dom}{\fun}{\divv{\funv}} + \I{\dom}{\funv\cdot \Deriv{\fun}}{\lem}
	\,.
\end{equation*}
\end{theorem}
\begin{remark}
	We think that the arguments in the proof can easily be
	adapted to prove the theorem for arbitrary open $\dom\csubset
	\encl{\dom}$.
\end{remark}
Note that the existence of the \tme{}s in the above theorem is trivial,
neglecting the \tcor{} and the support, $\ntfunvd{\funv}= \acme{\funv}{\lem}$ and
$\ntfunv{\funv} = \divv{\funv}$ would be viable choices. The difficulty lies in
the localisation of $\cor{\ntfunvd{\funv}} \subset \bd{\dom}$ and the support
of $\ntfunv{\funv}$.
\begin{proof}
	By Proposition \ref{prop:ntfs_sil} there exists a continuous linear
	functional $\ntfs{\funv}{\dom}$ on $\Lcspace{\bd{\dom}}$ such that for
	every $\fun \in \skpdr{1}{\infty}{\encl{\dom}}{\R}$
	\begin{equation*}
		\ntfs{\funv}{\dom} (\refun{\fun}{\bd{\dom}}) =
		\I{\dom}{\fun}{\divv{\funv}} + \I{\dom}{\funv\cdot
		\Deriv{\fun}}{\lem} \,.
	\end{equation*}
	Let $0 < \delta < \dist{\dom}{\bd{\encl{\dom}}}$. Note that by
	\cite[p. 131f]{Evans1992} every $\fun \in \skpdr{1}{\infty}{\encl{\dom}}{[0,1]}$
	is locally \tLcont{}. Since $\cl{\dom}$ is compact and path-connected, $\fun \in
	\Lcspace{\cl{\dom}}$ (cf. Lemma \ref{lem:bd_lip_path}). Then by Lemma
	\ref{lem:bd_lip_path} for every $\fun
	\in \skpdr{1}{\infty}{\encl{\dom}}{\R}$
	\begin{equation}\label{eq:ntfs_bd}
		\Lcnorm{\refun{\fun}{\bd{\dom}}}\leq
		\norm{C}{\refun{\fun}{\bd{\dom}}} + \const
		\norm{\infty,\dnhd{(\bd{\dom})}{\delta}}{\Deriv{\fun}}
	\end{equation}
	with $\const >0$ depending only on $\delta$ and $\bd{\dom}$.
	Note that
	\begin{equation*}
		\iota : \skpdr{1}{\infty}{\encl{\dom}}{\R} \to
		\Ccfun{\bd{\dom}} \times \liblnd{\dnhd{(\bd{\dom})}{\delta}}
		\text{ with } \iota(\fun) = \left(\refun{\fun}{\bd{\dom}},
		\refun{\Deriv{\fun}}{\dnhd{(\bd{\dom})}{\delta}}\right)
	\end{equation*}
	is continuous and linear. Set
	\begin{equation*}
		\BS_0 := \iota(\skpdr{1}{\infty}{\encl{\dom}}{\R}) 
	\end{equation*}
	and define
	\begin{equation*}
		\de_0 : \BS_0 \to \R \mspace{6mu}\text{ with }\mspace{6mu}
		\df{\de_0}{(\refun{\fun}{\bd{\dom}},\refun{\Deriv{\fun}}{\dnhd{(\bd{\dom})}{\delta}})}
		= \ntfs{\funv}{\dom}(\refun{\fun}{\bd{\dom}}) \,.
	\end{equation*}
	Then $\de_0$ is a continuous linear functional on the linear space $\BS_0 \subset
	\Ccfun{\bd{\dom}}\times \liblnd{\dnhd{(\bd{\dom})}{\delta}}$ by
	equation \eqref{eq:ntfs_bd}. 
	
	By the Hahn-Banach Theorem (cf.	\cite[p. 63]{dunford_linear_1988}) there exists a continuous linear
	extension $\de$ of $\de_0$ to all of $\Ccfun{\dom}\times
	\liblnd{\dnhd{(\bd{\dom})}{\delta}}$ with $\norm{}{\de}=
	\norm{}{\de_0}$. Note that the dual of a product space can be
	identified with the product of the dual spaces. Hence, as in Proposition \ref{prop:dual_libln}, there exist a \tRaM{}
	$\ntfunv{\funv}$ on $\bd{\dom}$ and a \tme{} $\me \in
	\bawln{\dnhd{(\bd{\dom})}{\delta}}$ such that for all $\fun\in
	\skpdr{1}{\infty}{\encl{\dom}}{\R}$
	\begin{equation}\label{eq:first_trace}
		\ntfs{\funv}{\dom}(\refun{\fun}{\bd{\dom}}) =
		\df{\de}{(\refun{\fun}{\bd{\dom}},\refun{\Deriv{\fun}}{\dnhd{(\bd{\dom})}{\delta}})}
		= \I{\bd{\dom}}{\fun}{\ntfunv{\funv}} +
		\I{\dnhd{(\bd{\dom})}{\delta}}{\Deriv{\fun}}{\me}
		\,.
	\end{equation}
	This proves that there is $\me \in \bawln{\encl{\dom}}$ with $\cor{\me}
	\subset \dnhd{(\bd{\dom})}{\delta}$ such that the above equation is
	satisfied. It remains to show that there exists $\ntfunvd{\funv}$ with
	$\cor{\ntfunvd{\funv}}\subset \bd{\dom}$ satisfying the same equation.
	Now, let
	\begin{align*}
		\BS_1 :=\{\til{\funv} \in
			\liblnd{\dnhd{(\bd{\dom})}{\delta}} \setpipe &\exists
			\fun \in \skpdr{1}{\infty}{\encl{\dom}}{\R}\\
			& \exists \funv \in \liblnd{\dnhd{(\bd{\dom})}{\delta}}: \\
			& \funv = 0 \text{ on }
			\dnhd{(\bd{\dom})}{\til{\delta}} \text{ for some } 0<
			\til{\delta} < \delta \\
			& \til{\funv} =	\Deriv{\fun} + \funv \}
	\end{align*}
	Then $\de_1 : \BS_1 \to \R$ with 
	\begin{equation*}
		\de_1(\til{\funv}) :=
		\I{\dnhd{(\bd{\dom})}{\delta}}{\Deriv{\fun}}{\me}
	\end{equation*}
	defines a linear functional on $\BS_1$ with
	\begin{equation*}
		\norm{}{\de_1} \leq \norm{}{\me}\,.
	\end{equation*}
	First, it is shown that the definition is independent of the
	decomposition of $\til{\funv}$. Therefore, let $\til{\funv}\in \BS_1$ and $\fun_1,\fun_2 \in
	\skpdr{1}{\infty}{\encl{\dom}}{\R}$, $\funv_1,\funv_2 \in
	\liblnd{\dnhd{(\bd{\dom})}{\delta}}$ be such that for some $0 <
	\til{\delta}< \delta$
	\begin{equation*}
		\funv_1 = \funv_2 =0 \text{ on }
		\dnhd{(\bd{\dom})}{\til{\delta}} 
	\end{equation*}
	and 
	\begin{equation*}
		\til{\funv} = \Deriv{\fun_1} + \funv_1 = \Deriv{\fun_2} +
		\funv_2 \,.
	\end{equation*}
	Then
	\begin{equation*}
		\Deriv{\fun_1} = \Deriv{\fun_2} \text{ on }
		\dnhd{(\bd{\dom})}{\til{\delta}} \,.	
	\end{equation*}
	Since $\bd{\dom}$ is path-connected, $\dnhd{(\bd{\dom})}{\til{\delta}}$ is path-connected. Hence $\fun_1
	- \fun_2$ is constant on
	$\dnhd{(\bd{\dom})}{\til{\delta}}$. Note that for $\til{\const}\in \R$
	\begin{equation*}
		\ntfs{\funv}{\dom}(\til{\const}) =
		\I{\bd{\dom}}{\til{\const}}{\ntfunv{\funv}} +
		\I{\dnhd{(\bd{\dom})}{\delta}}{0}{\me} =
		\I{\bd{\dom}}{\til{\const}}{\ntfunv{\funv}} \,.
	\end{equation*}
	Hence Equation \eqref{eq:first_trace} yields
	\begin{equation*}
		\I{\dnhd{(\bd{\dom})}{\delta}}{\Deriv{(\fun_1-\fun_2)}}{\me} =
		\ntfs{\funv}{\dom}(\refun{(\fun_1 - \fun_2)}{\bd{\dom}}) - \I{\bd{\dom}}{\fun_1 -
		\fun_2}{\ntfunv{\funv}} = 0 \,.
	\end{equation*}
	This shows that $\de_1$ is well-defined.

	Since $\BS_1 \subset \liblnd{\dnhd{(\bd{\dom})}{\delta}}$ is a linear
	subspace, the
	Hahn-Banach Theorem (cf. \cite[p. 63]{dunford_linear_1988}) yields an
	extension of $\de_1$ to all of $\liblnd{\dnhd{(\bd{\dom})}{\delta}}$
	and by Proposition \ref{prop:dual_libln} a \tme{} $\me_\funv\in
	\bawln{\dnhd{(\bd{\dom})}{\delta}}$ with
	\begin{equation*}
		\df{\de_1}{\til{\funv}} =
		\I{\dnhd{(\bd{\dom})}{\delta}}{\til{\funv}}{\ntfunvd{\funv}}
		\,.
	\end{equation*}
	By definition, 
	\begin{equation*}
		\I{\dnhd{(\bd{\dom})}{\delta}}{\funv}{\ntfunvd{\funv}}=\df{\de_1}{0
		+ \funv}=	0 
	\end{equation*}
	for $\funv \in \liblnd{\dnhd{(\bd{\dom})}{\delta}}$ with
	$\funv = 0$ on $\dnhd{(\bd{\dom})}{\til{\delta}}$ for some $0 <
	\til{\delta} < \delta$.
	Hence,
	\begin{equation*}
		\cor{\ntfunvd{\funv}} \subset \bd{\dom} \,.
	\end{equation*}
	Since for every $\fun \in \skpdr{1}{\infty}{\encl{\dom}}{\R}$
	\begin{equation*}
		\I{\dnhd{(\bd{\dom})}{\delta}}{\Deriv{\fun}}{\me} =
		\df{\de_1}{\Deriv{\fun} + 0} = 
		\I{\dnhd{(\bd{\dom})}{\delta}}{\Deriv{\fun}}{\ntfunvd{\funv}}
	\end{equation*}
	by definition, the statement of the theorem follows.
\end{proof}
\begin{remark}
	Note that the \tme{} $\ntfunvd{\funv}$ is a direct result of the
	analysis. In regular settings, this measure is expected to be zero (see
	also the following examples). For $\funv \in \dmid{\encl{\dom}}$
	and open $\dom\in \bor{\encl{\dom}}$ having finite perimeter such that
	the inner normal measure exists (see Example \ref{ex:inner_port}), Proposition
	\ref{prop:ntfunv_dmi} and Proposition \ref{prop:ram_on_core} yield the
	existence of a \tRaM{} $\ntfunv{\funv}$ on $\bd{\dom}$ such that for all compactly
	supported continuous
	functions $\fun \in \BVd{\encl{\dom}}$
	\begin{equation*}
		\divv{(\fun \cdot \funv)}(\dom) = \I{\bd{\dom}}{\fun}{\ntfunv{\funv}}
		\,.
	\end{equation*}
	In particular, $\ntfunvd{\funv} = 0$.

	For the general case note that for $k \in \N$
	\begin{equation*}
		\nap_k := \ind{\dnhd{(\bd{\dom})}{\frac{1}{k}}} (1-k
		\distf{\bd{\dom}}) \in \skpdr{1}{\infty}{\encl{\dom}}{\R} \,.
	\end{equation*}
	Since $\nap_k = 1$ on $\bd{\dom}$, Proposition \ref{prop:lip_bd_zero}
	yields
	\begin{equation*}
		\I{\dom}{\fun \nap_k}{\divv{\funv}} + \I{\dom}{\funv\cdot
		\Deriv{(\fun \cdot \nap_k)}}{\lem} =
			\I{\dom}{\fun}{\divv{\funv}} +
			\I{\dom}{\funv\cdot\Deriv{\fun}}{\lem} \,.
	\end{equation*}
	Using Dominated Convergence (cf. \cite[p. 20]{Evans1992}) yields
	\begin{align*}
		\I{\dom}{\fun \nap_k}{\divv{\funv}} & \xrightarrow{k \to \infty}
		0 \\
		\I{\dom}{\funv\cdot \Deriv{(\fun\nap_k)}}{\lem} & =
		\I{\dom}{\nap_k \funv \Deriv{\fun} + \fun \funv
		\Deriv{\nap_k}}{\lem} \\
		& \xrightarrow{k \to \infty} 0 + \lim \limits_{k \to \infty}
		\I{\dom}{\fun\funv\Deriv{\nap_k}}{\lem} \,,
	\end{align*}
	where the last limit exists because the other addends tend to zero and
	their sum is constant. Hence
	\begin{equation*}
		\lim \limits_{k \to \infty}
		\I{\dom}{\fun\funv\Deriv{\nap_k}}{\lem} =
		\I{\bd{\dom}}{\fun}{\ntfunv{\funv}} +
		\sI{\bd{\dom}}{\Deriv{\fun}}{\ntfunvd{\funv}} \,.
	\end{equation*}
	Note that the left-hand side is essentially the same as in Schuricht
	\cite[p. 534]{schuricht_new_2007} (cf. \cite[p. 449]{silhavy_divergence_2009}). 
\end{remark}
\begin{change}

The following examples proves, that $\ntfunvd{\funv}$ can be non-zero.
The function in the following example is the same as in
\cite[p. 449f]{silhavy_divergence_2009}.
\begin{example}
	Let $n=2$ and $\encl{\dom} = \ball{0}{2}\subset \R^2$. Furthermore, let
	$\dom = (0,1)^2$ and $\funv \in \dmpd{1}{\encl{\dom}}$ be defined via
	\begin{equation*}
		\funv(\eR,\eRR) := \frac{1}{\eR^2 + \eRR^2} \begin{pmatrix}
			\eRR \\ -\eR
		\end{pmatrix} \,.
	\end{equation*}
	Note that $\divv{\funv}$ is the zero \tme{}. In order to see this, let $\sfun \in
	\Conec{\encl{\dom}}$. Then
	\begin{align*}
		\I{\encl{\dom}}{\funv\cdot \Deriv{\sfun}}{\lem} & = \lim
		\limits_{\delta \downarrow 0} \I{\encl{\dom}\setminus
		\ball{0}{\delta}}{\funv \cdot \Deriv{\sfun}}{\lem} \\
		& = \lim \limits_{\delta \downarrow 0}
		\I{\bd{\ball{0}{\delta}}}{\sfun \funv \cdot
			\normal{}}{\ham^{n-1}} - \I{\encl{\dom}\setminus
		\ball{0}{\delta}}{\sfun \divv{\funv}}{\lem} \,.
	\end{align*}
	But $\funv \cdot \normal{} = 0$ on $\bd{\ball{0}{\delta}}$ and $\divv{\funv} =
	0$ on $\encl{\dom}\setminus \ball{0}{\delta}$. 

	Now, set
	\begin{equation*}
		\fun_k := \ind{\left (\frac{1}{k},\infty\right) \times \R} +
		\ind{\left (0,\frac{1}{k}\right)\times \R} k \distf{\{0\}\times
		\R} \in \skpdr{1}{\infty}{\encl{\dom}}{[0,1]} \,.
	\end{equation*}
	Then
	\begin{equation*}
		\Deriv{\fun_k} = \ind{\left (0,\frac{1}{k}\right) \times \R} k
		e_1 \,.
	\end{equation*}
	By Theorem \ref{thm:ggt_ubound}, there exists a \tRaM{}
	$\ntfunv{\funv}$ on $\bd{\dom}$
	and $\ntfunvd{\funv} \in \bawln{\encl{\dom}}$ with
	$\cor{\ntfunvd{\funv}} \subset \bd{\dom}$ such that for $k\in \N$
	\begin{equation*}
		\I{\dom}{\funv \cdot \Deriv{\fun_k}}{\lem} +
		\I{\dom}{\fun_k}{\divv{\funv}} =
		\sI{\bd{\dom}}{\Deriv{\fun_k}}{\ntfunvd{\funv}} +
		\I{\bd{\dom}}{\fun_k}{\ntfunv{\funv}} \,.
	\end{equation*}
	But $\fun_k = 0$ on $\bd{\dom}$ and $\divv{\funv} = 0$.
	Hence
	\begin{equation*}
		\I{\dom}{\funv \cdot \Deriv{\fun_k}}{\lem} =
		\sI{\bd{\dom}}{\Deriv{\fun_k}}{\ntfunvd{\funv}} \,.
	\end{equation*}
	Furthermore
	\begin{align*}
		\I{\dom}{\funv\cdot \Deriv{\fun_k}}{\lem} & = \mI{\left
			(0,\frac{1}{k}\right)}{\I{(0,1)}{\frac{\eRR}{\eR^2+
	\eRR^2}}{\eRR}}{\eR} \\
	& = \mI{\left(0,\frac{1}{k}\right)}{\left [\frac{1}{2} \ln (\eR^2+
\eRR^2)\right]_0^1}{\eR} \\
	& = \mI{\left(0,\frac{1}{k}\right)}{\frac{1}{2}\ln\left(\frac{1}{\eR^2}
+1\right)}{\eR} \\
& \geq \frac{1}{2} \ln (k^2 + 1) \xrightarrow{k \to \infty} \infty \,.
	\end{align*}
\end{example}
The example above shows that $\ntfunvd{\funv}$ can actually be non-zero, if the
concentrations of the vector field $\funv$ are sufficiently large near
$\bd{\dom}$. Thus, $\ntfunvd{\funv}$ is indeed necessary for the
characterisation of the Gauß-Green formula.
\begin{example}
	Revisiting Example \ref{ex:dirac_mass}, let $n=2$ and $\encl{\dom} := \ball{0}{2}\subset \Rn$. Furthermore, let
	$\dom := (0,1) \times (-1,1)$ and $\funv \in \dmpd{1}{\encl{\dom}}$ be
	defined by
	\begin{equation*}
		\funv(\eR,\eRR) := \frac{1}{2\pi}\frac{1}{\eR^2 + \eRR^2} \begin{pmatrix}
			\eR \\ \eRR 
		\end{pmatrix} \,.
	\end{equation*}
	Recall that $\divv{\funv} = \delta_0$.
	For $k \in \N$ let $\fun_k\in \skpdr{1}{\infty}{\encl{\dom}}{[0,1]}$ be
	defined by
	\begin{equation*}
		\fun_k := \ind{\left (\frac{1}{k},\infty\right )\times \R} + k
			\distf{\{0\}\times \R} \ind{\left
		(0,\frac{1}{k}\right)\times \R} \,.	
	\end{equation*}
	Then 
	\begin{align*}
		\I{\dom}{\funv\cdot \Deriv{\fun_k}}{\lem} & = \mI{\left
			(0,\frac{1}{k}\right)}{\I{(-1,1)}{\frac{\eR}{2\pi(\eR^2 +
		\eRR^2)}}{\eRR}}{\eR} \\
		& = \frac{1}{2\pi}\mI{\left(0,\frac{1}{k}\right)}{\left [
\arctan \frac{\eRR}{\eR}\right]_{-1}^1}{\eR} \\
& = \frac{1}{2\pi} \mI{\left (0,\frac{1}{k}\right)}{2\arctan\frac{1}{\eR}}{\eR}
\\
& \xrightarrow{k \to \infty} \frac{1}{2} \,.
	\end{align*}
\end{example}
In contrast to the previous example, this example shows a vector field
with a strongly concentrated divergence in zero, yet $\ntfunvd{\funv}$ seems to
be zero. Indeed, in \cite[p. 449]{silhavy_divergence_2009} Silhavy shows that
the normal trace can be represented by a \tRaM{}, if 
\begin{equation*}
	\lim \limits_{\delta \downarrow 0}
	\frac{1}{\delta}\I{\dom\setminus\idnhd{\dom}{\delta}}{|\funv \cdot
	\Deriv{\distf{\bd{\dom}}}|}{\lem} < \infty \,.
\end{equation*}
This holds true in the last example and thus $\ntfunvd{\funv}= 0$.
\end{change}
\dobib

\bibliographystyle{plain}
\bibliography{Diss.bib}

\end{document}